\documentclass[11pt]{article}
\usepackage{latexsym}
\usepackage{graphicx}
\usepackage{amsfonts}
\usepackage{amssymb}
\usepackage{amsmath}
\usepackage{ifthen}
\usepackage{lscape}
\usepackage{chngpage}
\usepackage{bibunits}
\usepackage{enumitem}
\usepackage{natbib}
\usepackage{endnotes}
\usepackage{color}

\let\footnote=\endnote

\definecolor{dgreen}{rgb}{0.1,0.6,0.3}
\definecolor{bluegreen}{rgb}{0.3,0.8,0.8}

\definecolor{bluered}{rgb}{0.6,0.3,0.4}

\usepackage{color}

\definecolor{orange}{rgb}{1,0.5,0}

\def\bbtau{\tau}
\usepackage{color}
\newtheorem{prop}{Proposition}
\newtheorem{thm}{Theorem}
\newtheorem{exm}{Example}
\newtheorem{dfn}{Definition}
\newtheorem{cor}{Corollary}
\newtheorem{lem}{Lemma}

\newtheorem{rem}{Remark}

\def\bq{\bar q}

\DeclareMathAlphabet{\mathcalligra}{T1}{calligra}{m}{n}

\def\tauu{\mathcalligra{t}}

\def\N{\mathbb N}
\def\R{\mathbb R}

\def\bTau{{{\tau}}}

\def\cI{\mathcal I}
\def\cR{\mathcal R}
\def\cL{\mathcal L}
\def\cP{\mathcal P}
\def\cW{\mathcal W}
\def\cS{\mathcal S}
\def\cO{\mathcal O}

\def\cA{\mathcal A}

\def\bft{\boldsymbol{\theta}}
\def\bfT{\boldsymbol{\Theta}}

\newcommand{\om}{\omega}

\newcommand{\bea}{\begin{eqnarray}}
\newcommand{\eea}{\end{eqnarray}}
\newcommand{\beas}{\begin{eqnarray*}}
\newcommand{\eeas}{\end{eqnarray*}}
\newcommand{\ba}{\begin{array}}
\newcommand{\ea}{\end{array}}
\newcommand{\be}{\begin{equation}}
\newcommand{\ee}{\end{equation}}

\def\a{\alpha}

\def\mf{\mathfrak{f}}
\def\md{\mathfrak{d}}

\def\cR{\mathcal R}
\def\cP{\mathcal P}
\def\cV{\mathcal V}

\def\bcR{\bar{\mathcal R}}

\def\N{\mathbb N}

\title{\LARGE Intertemporal Price Discrimination  with Time-Varying Valuations}

\author{{\large Victor F. Araman}
\\ {\footnotesize  Olayan School of Business, American University of Beirut, Beirut, Lebanon, va03@aub.edu.lb}\\ \large Bassam Fayad\\ {\footnotesize Institut de Math\'ematiques de Jussieu-Paris Rive Gauche, Paris, France, bassam.fayad@imj-prj.fr}}
\begin{document}
\maketitle
\date{}

\begin{abstract}

A firm that sells a non perishable product considers intertemporal price discrimination in the objective of maximizing its long-run average revenue. We consider a general model of patient customers with changing valuations. Arriving customers wait for a random but bounded length of time and purchase the product when its price falls below their valuation, which varies following a stochastic process. We show the optimality of a class of cyclic strategies and obtain an algorithm that yields them. When the pace of intertemporal pricing is constrained to be comparable to the upper bound on  customers patience levels, we have a good control on the cycle length and on the structure of the optimizing cyclic policies. The results also hold for forward looking strategic  customers as well as under an alternative objective of maximizing infinite horizon discounted revenue.

We cast our results in a general framework of optimizing the long-run average revenue for a class of revenue functions  that we denote {\it weakly coupled}, in which the revenue per period depends on a finite number of neighboring prices. We analyze in detail the case of Markovian valuations where we can obtain closed form formulations of the revenue function and some refinements of our main results. We also extend our results to the case of patience levels that are bounded only in expectation.

\end{abstract}

\section{Introduction}\label{Sec: Intro}
Typically, customers try to estimate the value they will generate from a product and then decide whether or not to purchase. This valuation exercise is quite complex as it may depend on many external factors and is often based on incomplete product information. As a result, it is not uncommon for customers to repeatedly change their willingness-to-pay, sometimes drastically, before eventually making the purchase or deciding to drop it completely. 

When customers are interested in a product, they often go through an \textit{evaluation} phase, in which they gather information about the product and reach a first estimate of their willingness-to-pay or valuation. They move next to the \textit{purchasing decision} phase. Some customers may instantly purchase the product if they find the current price adequate. Others simply delay their decision seeking a better match in the future. Those customers will then be exposed to additional information, be it about the product or their personal state, that could affect their valuation and thus their purchasing decision. It is this latter phase that we are interested in modeling and analyzing.

Consider for instance a customer who is looking to buy a new camera. The customer has an estimate of her willingness-to-pay based on a previous experience she had with that brand, some information gathered online, and the urgency to buy the product. Given this valuation, she can decide to buy the camera or delay the purchase depending on how attractive the price is and whether she is expecting, say, a better deal in the near future. If she delays the purchase, she might indeed witness changes in the prices but possibly also in her own valuation of the product. These changes in the valuation may be due to a number of factors including, new product/customers-reviews, marketing campaigns (e.g., targeted online advertising), availability and prices of competing products, and so on. The valuation can also  be affected by the customer's personal information: e.g., a customer is going on an unplanned vacation and for whom the camera becomes more valuable; another customer gets a bonus at work and can now allocate a larger budget for the camera. All this information may be received sequentially, feeding the decision process and resulting in a dynamic update of the valuation.

Such behavior has been generally overlooked in the pricing literature, where customers are most often assumed to have a reservation price (possibly drawn from some given distribution) constant throughout the purchasing process. In this work, we consider customers that are interested in purchasing a product and see their valuations change through time. We shed some light on the impact that such varying valuations have on the seller's optimal pricing policy.

Intertemporal pricing is experienced in practice and often argued for as a way to price discriminate, and take advantage of customers heterogeneity, scarcity of capacity, and demand stochasticity. In our case, we analyze this phenomenon in a specific context characterized by the following modeling choices.
We consider  a setting where a non-perishable product is sold throughout an infinite horizon during which the seller has \textit{committed} to a sequence of prices. A continuous influx of heterogenous customers become interested in the product and approach the seller at different times. They could purchase upon arrival or remain in the system for a (random) period of time during which their valuations \textit{change} stochastically. We disregard any inventory and cost-related issues and assume the product is available throughout. The fact that the firm commits to future prices is aligned with recent literature on intertemporal pricing with patient and/or heterogenous customers   (see, \cite{Board2008}, \cite{Besbes15}, \cite{LiuCooper15}, \cite{Lobel17}, \cite{Garrett16} and \citet{Yiwei_Far_Trich18}). Customers' heterogeneity is modeled through the dynamics of the stochastic valuation process governing each arrival.

The questions we are interested in are whether, in this context, there is an opportunity for price discrimination; and if so, how to go about characterizing it formally. We specifically examine the optimality of cyclic policies and how numerically efficient these policies are, especially since such cyclic policies have been proven to be optimal in similar settings in the literature.

\vspace{0.2cm}

Our main contributions are the following.
\begin{itemize}
\item[1.] We introduce a very general and innovative model that incorporates  time-varying valuations in a simple setting of intertemporal pricing for which we obtain structural results for the optimal pricing strategies.

\item[2.]  When customers' lifespans in the system are bounded, we cast our results in a general framework of optimizing the long-run average revenues for a class of revenue functions  that we denote {\it weakly coupled}, in which the revenue per period depends on a finite number of neighboring prices.  Within this framework, we show the optimality of cyclic policies. We characterize such solutions and present an algorithm to obtain them.

\item[3.] We introduce a control parameter $M$ for the firm that measures the number of price changes that any customer may witness. Our algorithm yields the optimal pricing policies  in polynomial time in the number of prices available, but exponential in $M$.  By constructing an example where the optimal policy is cyclic, \textit{increasing} and involves all but one  of the available prices, we argue that, in general, in such changing valuation setting, the exponential complexity is unlikely to be reduced.

\item [4.] Our algorithm becomes fully relevant for small values of $M$. When customers can only see at most two different prices (i.e. $M=1$), we show that the pricing policy is either a fixed-price policy or cyclic and simple i.e. each price in the cycle is set only once for exactly a specified amount of time.

\item[5.] We analyze in detail the case of Markovian valuations where we obtain closed form formulations of the revenue function and some refinements of our main results. The Markovian model paired with a simple intertemporal pricing setting allows us to argue that the complexity of the problem can be traced primarily to the sole property of changing valuations.
In realistic Markovian models, we confirm numerically that considering the  consumer's changing valuation behavior may have a major impact on the pricing policy that the firm should adopt. In particular, overlooking such behavior can cause a sizable opportunity cost.

\item[6.] We establish the robustness of the optimality of cyclic policies  by extending our results  in three important directions: $i)$ patience levels bounded only in expectation; $ii)$ forward-looking or strategic customers; $iii)$   infinite horizon discounted revenue instead of  long-run average revenue.

\end{itemize}

\subsection{Literature Review}\label{sec: Lit Review}

The literature on intertemporal pricing is quite extensive starting with the early work on durable goods (e.g., \cite{Coase72}, \cite{Stockey79}, \cite{Stockey81}, \cite{Conlisk84}, \cite{Besanko90} and \cite{Sobel91}).  
The work of \cite{Coase72}, \cite{Stockey79} and \cite{Stockey81} are among the first to argue that price commitment policies in the presence of strategic customers are an effective way to exercise market power. In their model, the firm  faces a population of customers that is present in the system from the start. In contrast, others (e.g., \cite{Conlisk84}, \cite{Besanko90} and \cite{Sobel91}) consider a firm that dynamically sets its price and face a continuous flux of strategic customers. The work of \cite{Conlisk84} is the closest in this literature to our work.  Customers in this paper have two possible valuations of the product, Low and High, and they remain in the system indefinitely until they purchase. Interestingly, \cite{Conlisk84} show that the optimal pricing policy is both cyclic and decreasing.

More recently, the operations management literature has looked at different variants of intertemporal pricing models (see, the recent reviews of \cite{Bitran03} and \cite{AvivVulc12}). Often in this literature, customers are myopic and impatient. They arrive through time, each endowed with a valuation drawn from a given distribution.  Customers purchase upon arrival if the price is higher than their valuation, otherwise they leave the system. Some of this literature incorporates customers' strategic behavior (see, the review of \cite{Shen07}).  \cite{Levin09} and \cite{Levin10} consider respectively a monopolistic and a set of firms selling to a set of differentiated customers segments and assume that all customers are in the system from the start. 
\cite{Aviv08} consider forward looking customers with declining valuations who arrive following a Poisson process and where the listed price changes only once. More closely related to our work is the paper of \cite{Su07} that considers a deterministic model in which strategic customers are differentiated not only through their valuation, but also, through their patience level. Optimal prices are shown to be monotone. The very recent work of \cite{Caldentey16} considers an intertemporal pricing problem under minimax regret where both strategic and myopic patient-customers are considered. 
Finally, two recent notes (\cite{Wang16} and \cite{Huetal16}) tackle the problem of intertemporal pricing in the presence of reference price effects in an infinite horizon setting with discounted revenue.

Besides that of \cite{Conlisk84}, the papers that are the most related to our model are the recent works of \cite{Besbes15}, \cite{LiuCooper15} and \cite{Lobel17}. Their models are variants of \cite{Conlisk84}, but, like us, consider a seller who commits to the pricing policy at the start of the horizon and maximizes the long-run average revenue. In the case of \cite{Besbes15}, heterogenous customers, through their valuations and their patience level, are strategic.  They prove that a cyclic pricing policy is optimal but does not need to be decreasing. In an almost exactly similar setting, \cite{LiuCooper15} consider patient customers who purchase as soon as the price drops below their willingness-to-pay. In their case, when the patience levels are deterministic and fixed, cyclic and decreasing policies are optimal. Finally, \citet{Lobel17} considers the exact setting of \cite{LiuCooper15} and suggests a dynamic programming algorithm that runs in polynomial time for finding optimal pricing policies under bounded, though arbitrary, distribution of the patience levels.

None of the above considers valuations that change through time. It is only recently that a handful of papers have tackled such behavior (see, \cite{Garrett16}, \cite{Deb14} and \cite{GallegoSahin10}). \cite{Garrett16} considers, similarly to us, an infinite horizon setting where customers arrive through time with valuations that stochastically change while facing a committed pricing path. Though our work is different in a number of ways. First, the model is set in continuous time where both buyers and sellers have a common discount rate and where customers are strategic and homogeneous with a valuation governed by a Markov process that continuously switches between only two values, Low and High.  The seller is optimizing on the set of pricing paths (on the entire positive line).  The optimal pricing path is shown to be cyclic decreasing (\`{a} la \cite{Conlisk84}). As opposed to \cite{Garrett16}, we consider a very general valuation process that is only required to be stationary.  We do consider the case where the process is Markovian, however we do not restrict it to taking two values (except for the simplified model of Appendix~B). Moreover, we assume that prices belong only to a finite and discrete set. Another recent paper that incorporates changing valuations is that of \cite{Deb14}. The model considers one unit of a good and one buyer who is present in the system from the start and has a valuation drawn from some distribution. At some random time, a stochastic shock occurs and causes the buyer's valuation to be drawn again (from the same initial distribution). Finally, we mention the paper of \cite{GallegoSahin10} that models changing valuations in the context of revenue management. They consider a firm selling a finite number of units of a product that can be purchased anytime during a finite horizon and get consumed at the end. Customers are present at the beginning of the horizon and see their expected valuations changing through time. The firm commits to the prices and the customers need to decide when to purchase.

\subsection{Outline of the Paper}
In Section~\ref{sec: simple} we briefly introduce a Markovian model of changing valuations. The objective of this section is to highlight, through a simple setting of an intertemporal pricing problem, the complexity that the moving valuations property  brings and compare that to the literature that focused on constant valuations. In Section~\ref{sec: Model_Main}, we introduce a very general model of changing valuations. We formulate an open loop optimization problem and state the main results with respect to the optimality of cyclic policies, their structure, and how efficiently these can be obtained.   In Section \ref{sec: Weakly Coupled}, we introduce the notion of \textit{weakly coupled} revenue functions and show that the revenue function of a pricing policy under a general but bounded valuation model belongs to the class of  weakly coupled functions. We next characterize the solutions of the optimization problem, defined in the previous section, when applied to this class of revenue functions. We also present an algorithm to obtain them. The proofs of our general results on optimal policies follow in a straightforward manner.
In Section~\ref{sec: Markovian_Model}, we focus on the special case in which the valuations processes follow discrete time Markov chains. In this context, we obtain a closed form for the revenue function. We also show in this setting that the revenue function has an affine structure that allows us to refine some of our main results.  In Section~\ref{sec: Complexity}, we discuss the complexity of the optimal policy.  To support our claim that, in general, the optimal policy lacks any simplifying structure, we give two examples, one for general weakly coupled functions and one for the specific setting of Markovian valuations. Consequently, we argue that our results are valuable in obtaining efficient approximations of the optimal policy while preserving the changing valuations feature. In Section~\ref{sec: extensions}, we analyze some important extensions. First, we consider the case where the patience levels are bounded only in expectation and show, in the Markovian setting, that all the results hold  when optimality is replaced with $\varepsilon$-optimality. Next, we discuss how the main results can be adjusted to the case in which the firm is maximizing an infinite horizon discounted revenue function. Finally, we argue how our results extend to the case where customers are forward looking. We devote  Section~\ref{sec: Numerics} for an extensive numerical analysis that takes full advantage of the algorithm suggested earlier,  and shows the opportunity cost of discarding the changing valuations property. We conclude in Section~\ref{sec: Conclusion}. In Appendix~A, we relate the proofs of the Markovian model, in particular, when the revenue function  is affine.  In Appendix~B, we analyze the setting where the firm can set only two prices and the valuation process is Markovian with unbounded patience levels. We specifically discuss in this setting the existence of a reset time.

\section{A Simple Time Varying Valuation Model} \label{sec: simple}
The objective of this section is to introduce a simple model of intertemporal pricing in an infinite horizon setting and highlight the complexity of our model's changing valuation feature. For this reason, we do not precisely define some of the concepts (e.g., cyclic policy) until the next section. There, we carefully introduce a general model, as well as the different definitions needed for the analysis, and  rigorously state our main results.

Consider a stream of  customers interested each in buying one unit of a non perishable product. At each period $t\in\N$, one customer arrives to the system.  If customers don't purchase on arrival, we assume in this simple model that they all have a lifespan of $\tau$ periods, beyond the period of their arrival, to purchase the product. This window of time is known as the patience level.  If customers don't purchase by the end of their lifespan, they leave the system (if $\tau=0$, all customers are impatient and they either purchase on arrival or leave without purchasing). 

Customers arriving to the system have a valuation $v\in\Omega=\{v_1,v_2,..,v_K\}$ with $K<\infty$ which is drawn following some distribution ${\boldsymbol\gamma}$. During the ``lifespan" in the system, each customer's valuation of the product changes following a simple Markov chain with a given transition matrix $\bar Q$ on $\Omega$.

The seller adopts a committed pricing policy, whereby at time 0, the prices are set for the entire horizon, $\pi=(p_1,p_2,..)$, as a way to maximize the long-run average revenue. As mentioned earlier, the fact that the firm commits to future prices is aligned with recent literature on intertemporal pricing and is shown to be the optimal strategy to follow  (see, \cite{Board2008}, \cite{Besbes15}, \cite{LiuCooper15}, \cite{Lobel17}, \cite{Garrett16} and \citet{Yiwei_Far_Trich18}). We denote by $L_t(p_1,..,p_t)$  the revenue  function, which is the expected revenues generated during the first $t$ periods, having set the first $t$ prices  to be $p_1,...,p_t$. Let $\cR(\pi)=\lim\sup_{t\rightarrow\infty}\frac{1}{t}\,L_t(p_1,...,p_t)$ be the long-run average revenue generated by a policy $\pi$. Without loss of generality, we assume that the $p_i$'s belong to $\Omega$. We denote by $\cP$ the set of all possible pricing policies. The seller is looking to solve for
\begin{equation}\sup_{\pi\in\cP}\cR(\pi).\label{Eq: General_Main_Form}\end{equation}

The closest works in the literature to the model introduced above are those of \cite{Besbes15} and \cite{LiuCooper15}. But, for these papers, customers keep the same valuation (drawn on arrival) during their lifespan, which is equivalent to setting $\bar Q$ equal to the identity matrix, $\bar Q=\cal I$. Customers are strategic in the case of \cite{Besbes15}, while in the case of \cite{LiuCooper15}, customers are patient where they purchase as soon as their current valuation is larger than the current price. The results are of similar nature whether customers are forward-looking or patients. For now, we focus mainly on patient customers and address forward-looking customers in the extensions.

\subsection{Optimality of Cyclic Policies}

As an example of the results obtained in the literature addressing non-varying valuations, we now recall, using our notations, the main result of \cite{LiuCooper15}.
\begin{thm}[\cite{LiuCooper15}] \label{prop: Cooper} If $\bar Q={\cal I}$, then there exists a decreasing cyclic policy that solves (\ref{Eq: General_Main_Form}) which cycle size is smaller or equal to $K+\tau-1$ and can be obtained in polynomial (in $K$ and $\tau$) elementary computations.
\end{thm}

A similar result has also  been obtained in \cite{Besbes15}, except that they showed in their strategic setting that optimal policies are not necessarily decreasing, but satisfy the so-called \emph{reflection principle}. With our simple Markovian model for the valuation process, we obtain that
\begin{thm}\label{prop: Th1} For a given transition matrix $\bar Q$, there exists a cyclic policy that solves for (\ref{Eq: General_Main_Form}). The cycle size of the policy is in the order of ${\cal O}(K^\tau)$. The policy can be obtained through an algorithm that requires ${\cal O}(K^{4\tau})$ elementary computations.
\end{thm}

First, our result confirms that cyclic policies remain optimal even under changing valuations. As for the cycle size and the complexity of finding the optimal policy, we note that for  $\tau=1$, our results and those of \cite{LiuCooper15}
are comparable. However, for larger $\tau$, Theorem~\ref{prop: Th1} cannot guarantee a good performance of the algorithm that generates the pricing policy. In \cite{LiuCooper15}, the optimal pricing policy remains polynomial in $\tau$. Next, we discuss these points in some detail.

\subsection{The Structure of Optimal Cyclic Policies and the Complexity of the Optimization}

The recent literature, as reflected in Theorem~\ref{prop: Cooper}, recognizes that cyclic policies are optimal in some variants of our current setting as long as $\bar Q =\cal I$, (see also, \cite{Ahnetal07} and \cite{Besbes15}). Our main results and specifically Theorem~\ref{prop: Th1} are confirming the optimality of cyclic policies even when $\bar Q\neq \cal I$. However in this context of time-varying valuations, the proof of our result requires a more complex approach.

In order to prove that optimal policies are cyclic and tractable, all the aforementioned papers relied on a regenerative point argument. Indeed, it is observed that for any pricing policy, the   essential minimum price (the lowest price that appears infinitely many times in the policy) acts as a regenerative point. The optimization problem of choosing the prices in between  each two appearances of the minimum price is the same, which yields the optimality of  cyclic policies.

This regenerative point argument does not apply when valuations change through time, even in the simple Markovian setting of Theorem \ref{prop: Th1}. Given a pricing policy, customers with lower valuations may have higher valuations in the future; hence the regenerative points have no reason to exist \textit{even} for the optimal policy.

The regenerative point argument also helps to show interesting structural results on the optimizing cyclic policies in the constant valuations models.
For instance, Theorem~\ref{prop: Cooper} asserts that optimal cyclic policies are decreasing. In fact, this decreasing property of the policy allows \cite{LiuCooper15} to show that the optimal policy can be obtained in polynomial time. Recently, \citet{Lobel17} generalized this latter result to customers having constant valuations and  different patience levels. He observed that in this more general setting,  the optimal cyclic policy may fail to be decreasing;  it remains made of subcycles, each of which is \textit{decreasing}. In \cite{Besbes15}, the reflection principle property gives the tractability of the optimal cyclic solutions.

The next proposition shows how, even in the simplest Markovian setting, a changing valuation model does not allow optimal cyclic policies to inherit any of the \emph{nice} structures  of optimal policies obtained in the constant valuation case.

\begin{prop}  \label{prop: Example_basic_intro} Let $\tau:=1$. For any $K \geq 5$, there exists ${\bf v}=(v_{1},\ldots,v_K)$, and $\boldsymbol\gamma=(\gamma_1,...,\gamma_K)$ and a transition matrix $\bar Q$, such that  the optimal solution is cyclic, increasing in its prices, and with a cycle size equal to $K-1$.
\end{prop}


We give the explicit example that proves Proposition~\ref{prop: Example_basic_intro}  in section~\ref{sec: Complexity}. This proposition shows that despite customers spending at most two periods in the  system ($\tau=1$), the optimal cyclic policy can have a cycle as large as $K-1$. We also recall that when both valuations are constant and all customers have one patience level, \citet{LiuCooper15} observed numerically and conjectured that under some broad conditions the cycle size is independent of $K$ and depends only on the patience level. In the case of varying valuations, it seems that the optimal cyclic policy can be of any size.

The other surprising feature of the example of Proposition~\ref{prop: Example_basic_intro}, which is in full contrast with the results in the literature pertaining to constant valuations, is  that the prices in a cycle of an optimal policy are increasing. 

In conclusion, in the presence of stochastic valuations (even in the basic Markovian model), the optimal policies may fail to have any of the  nice structures of optimal policies in the constant valuations models, that allowed to reduce the complexity of the optimization problem.  With the lack of simplifying structure, it is hard to imagine how one can reduce the complexity of obtaining the cycle of the policy and thus, the curse of dimensionality seems inherent to this type of intertemporal pricing problem.


\subsection{Weakly Coupled Revenues}

An influential work on intertemporal pricing for durable goods is \cite{Conlisk84} that proved the optimality of cyclic policies.  The approach followed by \cite{Conlisk84} is based upon customers' accumulation type argument. The accumulation argument relies on the fact that customers with the lowest valuation who don't purchase remain in the system indefinitely, making it eventually worth resetting the system by setting the price at the lowest valuation for one period. This approach does not apply in our case where the expected number of customers in the system remains uniformly bounded, preventing any accumulation of customers in the system. In Appendix~B, in the case of $K=2$ under a Markovian setting, we obtain necessary and sufficient conditions under which the seller is better off setting a reset price at the end of the cycle, therefore showing that cyclic policies are strictly optimal. When $K>2$, it does not hold that, in general,  optimal (or near optimal) cyclic policies  contain a reset price.

To address this, we introduce in Section \ref{sec: Weakly Coupled}  the concept of weakly coupled revenue functions in which the revenue per period depends upon a finite number of neighboring prices. This property generalizes to some extent the concept of a regenerative point, and allows us to
establish the optimality of cyclic policies, as well as to get  an algorithm that detects them.

We now move to introduce a general model of time varying valuations and state the corresponding results.

\section{The General Bounded Patience Level Model, or BP Model.}\label{sec: Model_Main}
\subsection{Model description}\label{sec: Model}
We assume that at each period $t\in\N$, a random number $\xi^i_t\in\N$ of customers arrive to the system where $i\in \cI$ indicates the customer's class, and $\cI$ is a finite set of integers.  Customers either decide to purchase on arrival or remain in the system. During the ``lifespan" in the system, each customer's valuation of the product changes following a stochastic process. 
All customers have a {\it patience level},  which is a maximum window of time during which they are willing to wait
to purchase, before they lose interest in the product. We assume that the patience level of all customers is bounded by some $\tau \in \N \setminus \{0\}$. Customers may also lose interest in the product before time $\tau$. This is modeled by the existence of an absorbing state $0$ for the valuation processes. 

More precisely, our assumptions on the arrival process are the following:
\begin{itemize}
\item For any $i\in\cI$, $\xi=(\xi^i_t:t\geq 0)$ is a stationary process in $t$, such that $\sup_{i\in\cI}\mathbb{E}\xi_0^i<\infty.$ 
\item For any $i\in\cI$, $t\geq 0$, $1\leq k\leq \xi^i_t$ , the process  $(\cV_t^{i,k}(s):s\geq 0)$ is the valuation process that takes values on the non-negative real line and is independently and identically generated for each customer $k$ of class $i$ that arrives at time $t$. Here $s$ is the time after the instant $t$, with $\cV_t^{i,k}(0)$ being the valuation on arrival. We define the sequence of processes $\cV^i_t=(\cV_t^{i,k}(\cdot):1\leq k\leq \xi_t^i)$ and $\cV^i\equiv\left(\cV_t^i:t\geq 0 \right)$ is assumed to be a stationary process (in $t$).
\item There exists a positive integer preset constant $\tau$ such that for any $i\in\cI$, $t\geq 0$, and $1\leq k\leq \xi_t^i$, $\cV^{i,k}_t(s)=0$, for all $s\geq \tau$. Moreover, if $\cV^{i,k}_t(s)=0$, then for all $s'\geq s$, $\cV^{i,k}_t(s')=0$.
\end{itemize}

This general model will be denoted hereafter as the (BP) model, highlighting the Bounded Patience feature. In Section~\ref{sec: Markovian_Model}, we discuss in more detail the Markovian valuation case under bounded patience levels, denoted (M-BP) which was initially introduced in Section~\ref{sec: simple}; in Section~\ref{sec: extensions}, we discuss the Markovian valuation case under \textit{unbounded} patience levels, denoted (M-UP).

We assume for the main exposition that customers are patient and will purchase the product as soon as their valuation reaches a value higher than the current price. If that does not happen on arrival, or during their patience level, they do not purchase and leave the system. In the extension section (see, Section~\ref{sec: extensions}), 
we show that most of our results generalize to the case where customers are \textit{strategic}.

On the supply end, we recall that the seller commits to an infinite sequence of prices at the beginning of the horizon that belong to a finite set of values $\Omega$, where $\Omega:=\{v_{j}:1\leq j\leq K\}$, with $0< v_1< v_2< ...<  v_{K}$. To ease the notational burden we define the infinite price sequences $\pi=(p_t:t\geq 0)\in\cal P$ with $p_t \in [1,K]$ denoting the index of the price at time $t$. Hence, $p_t=k$ means that at time $t$ the value of the price is set at $v_k$.


Observe that the primitives of the model are given by the pair $(\xi^i,\cV^i)_{i\in\cI}$. Now, given such a pair, we are again looking to maximize the corresponding long-run average revenue, $\cal R$ among all policies $\pi=(p_1,p_2,...)\in\cP$,  i.e., solve for problem (\ref{Eq: General_Main_Form}). 
We say that $\pi^*$ is optimal if it is a solution to (\ref{Eq: General_Main_Form}).

The results of this paper can also be extended beyond the long-run average revenue type-objective on which the related literature has focused on. For example, we show in the extensions, Section~\ref{sec: extensions}, that most of our main results hold for the case of an infinite horizon discounted revenue function, where $\cR(\pi)$ is of the form $\sum_{t=0}^\infty e^{-rt} {L_t}(p_1,...,p_t)$.


\subsection{Main Results}\label{sec: Main results}

In this section, we state our main results on the optimality of a class of cyclic policies in solving problem (\ref{Eq: General_Main_Form}), together with an algorithm that finds these optimal policies with a complexity that is logarithmically  small compared to the {\it a priori} number of cyclic solutions in our class. Moreover, with an  additional constraint on what we will define as the firm's {\it pricing pace}, these optimal solutions become effectively tractable.

We start with a few definitions and notations about cyclic policies and pricing pace, then state our results in a unified way for the unconstrained and the pace-constrained cases. 

\begin{dfn}{\bf (Cyclic policies)}  \label{def:simple}

\noindent A policy $\pi=(k_1,k_2,\ldots)$ is called \textit{cyclic} if there exists $n \geq 1$ such that
$k_{j+n}=k_j$ for all $j \in \N$. When $n$ is the minimal integer with the latter property, we say that $(k_1,k_2,..,k_n)$ is the cycle of $\pi$, $n$ its size, and denote $\pi$ by its cycle $\pi=(k_1,\ldots,k_n)$.
\end{dfn}
A fixed-price policy, one in which the same price is set indefinitely, can also be viewed as cyclic with $n=1$.

\noindent At this point, we introduce a way to express pricing policies that will be convenient for our exposition. Any pricing policy $\pi\in\mathcal{P}$ can  be viewed as a sequence of \textit{phases} of the form $\pi=((k_j,\tau_j):j\geq 1)$, where each  $(k_j,\tau_j)\in\Omega\times \mathbb{N}\cup\{\infty\}$ is called a phase (of the policy) with $k_j$ being the price of the $j^{th}$ phase and $\tau_j$ its duration, which is the time during which the price is continuously set at $k_j$. 
For all $j\geq 1$, set $T_{j}:=\tau_1+\tau_{2}+...+\tau_{j}$. Note that this expression of the policy is not unique unless we impose that $k_{j}\neq k_{j+1}.$

We next introduce the  notion of a pricing pace. This is a constraint on the prices that the firm can set to regulate how frequently a price changes through time.  
The pace will be regulated by a constant  $\sigma \geq 1$, defined as the minimal duration that the firm allows each price to be set for.  From a practical point of view, the existence of such pricing pace seems quite natural due to the direct and indirect costs incurred by frequently changing prices, whether from the logistical end, or, as a reflection of customers' dissatisfaction. As a result of such pace, the set of admissible policies will be reduced.

\begin{dfn} {\bf ($\sigma$-policies)} \label{def:Msimple}
For any $\sigma \geq 1$, we denote by $\mathcal{A}_{\sigma}$ the set of policies where each phase has a duration that is a multiple of $\sigma$. We can always write such policy as a sequence of $\sigma$-phases of the form  $\pi=((k_j,\sigma):j\geq 1)$, where $k_{j+1}$ can be equal to $k_j$. Policies in  $\mathcal{A}_{\sigma}$  will be called $\sigma$-policies.
\end{dfn}

\noindent When the pricing pace is set to $\sigma$, and if we restrict ourselves to policies where prices are set in multiples of $\sigma$,
the firm's problem becomes a constrained optimization over $\cA_\sigma$
\begin{equation}\sup_{\pi\in\cA_\sigma}\cR(\pi).\label{Eq: General_Pace_Form}\end{equation}

 \noindent Note that the special case of $\sigma=1$ corresponds to the unconstrained optimization on $\cP$, as initially defined in (\ref{Eq: General_Main_Form}).

We define the quantity $M:=\lceil\tau/\sigma\rceil$. By setting a minimum pace $\sigma$, the firm can control the maximum number of price changes, $M$, customers could witness during their lifespan in the system. For example, if $\sigma$ is set so that $M=2$, customers will see one price on arrival and at most two \textit{additional} prices during their patience level. Up to enlarging $\tau$ if necessary, we will assume all through the paper that
$$\tau=M\sigma.$$

 In the unconstrained case $\sigma=1$, and hence $M=\tau$.

\noindent Before we state our results we introduce the notion of  $M$-simple cyclic policies.
 \begin{dfn} {\bf ($M$-simple)} For $M \in \N\setminus\{0\}$, a cyclic policy  $\pi=((k_j,\sigma):j\geq 1)\in {\mathcal A}_{ \sigma}$ is said to be $M$-simple if any string of $M$ prices $(k_{j+1},\ldots,k_{j+M})$ corresponding to $M$ consecutive $\sigma$-phases of the policy appears at most once during a cycle.
For $M=1$, we just say the policy is simple (i.e., all possible values of the prices appear at most one time in a cycle during exactly $\sigma$ consecutive periods).
\end{dfn}



\begin{thm}~\label{thm: Main} There exists an  $M$-simple cyclic  $\sigma$-policy
 that optimizes $\cR$ on  $\mathcal{A}_{\sigma}$.
 In particular, there exists a $\tau$-simple   cyclic policy
 that optimizes $\cR$ on  $\mathcal{P}$.
\end{thm}

We move now to discuss how computationally meaningful this result is.  Observe that an $M$-simple policy in $\mathcal{A}_{\sigma}$ has a cycle made of at most $K^M$  $\sigma$-phases. One can show\footnote{Consider a cyclic policy $\pi\in{\cal A}_\sigma$, so $\pi=((k_j,\sigma): j\leq J)$ for some $J\leq K^M$. WLOG we assume that $J$ is a multiple of $M$ and hence $(k_j:1\leq j\leq J)$ can be seen as a sequence of $M$-strings $(W_1~W_2~...~W_i)$, with $i\leq K^M/M$. We have $K^M$ choices of an $M$-string, $W_1$.  The next string, $W_2$ needs to be picked among $K^M-1$, $M$-strings. Hence the choice of $W_1W_2$ gives $K^M (K^{M}-1)$ possibilities. The third string $W_3$ must avoid all the $M$-strings contained in $W_1W_2$. The number of such $M$-strings is $M$. Hence, the number of possibilities for $W_3$ is $K^M-M$. For $W_4$ the possibilities are $K^M-2\,M$ and hence, recurrently, we obtain that the number of simple policies of the form $W_1W_2\ldots W_i$ is  $K^M (K^{M}-1) (K^{M}-M)\ldots (K^{M}-(i-2)M)/i$. We divide by $i$ to eliminate the permutations due to the fact that the policy is cyclic.  The number of $M$-simple policies in ${\cal A}_\sigma$ is  $\sum_iK^M (K^{M}-1) (K^{M}-M)\ldots (K^{M}-(i-2)M)/i$ which is in the order of the last contribution of this sum, $K^M (K^{M}-1) (K^{M}-M)\ldots (K^{M}-(i-2)M)/i$ with  $i=K^M/M$.  By multiplying and dividing by $M^{i-2}$ and replacing $i$ by its value the previous term is shown to be in order of $M^{K^M/M}\,(K^M/M)!$, which is by applying Stirling formulae, in the order of $K^{K^M}$.} that the total number of $M$-simple policies in ${\cal A}_\sigma$ is of the order of $K^{K^M}$. The next result shows that this \textit{a priori} double-exponential complexity can be logarithmically reduced by an adequate algorithm that computes an optimal strategy of Theorem~\ref{thm: Main}. The complexity remains exponential in $M$.


\begin{thm}\label{thm: Algo}
The optimal cyclic policies of Theorem~\ref{thm: Main} can be obtained, through an algorithm that requires $\cO(K^{4M})$ elementary computations.
\end{thm}

\noindent Theorems~\ref{thm: Main} and \ref{thm: Algo} will be proved in the next two sections by relying on Proposition~\ref{propweakly} and respectively, Proposition~\ref{cor: Weakly Coupled} and Corollary~\ref{cor: algo}.

Theorem~\ref{thm: Algo} shows that the computational complexity of obtaining the optimal policy is affected polynomially in the number of possible prices $K$. Hence,
the problem is tractable for small values of $M$, while for $M$ large, despite a major reduction, it remains exponentially complex. We believe that in the context of changing valuations one cannot reduce much further the order of complexity. We revisit this claim in Section~\ref{sec: Complexity}.

Next, we cast our main results in a general framework of optimizing the long-run average revenue for a specific class of revenue functions. 


\section{Weakly Coupled Revenue Functions }\label{sec: Weakly Coupled} 


In this section, we show that the revenue function corresponding to a (BP) model belongs to  a general class of revenue functions that we denote weakly coupled. In the latter framework, it is shown that the optimization problems formulated in (1) and (2) admit cyclic policies. We also provide an algorithm that detects the optimal cyclic policies. With respect to the latter, finding the optimal cyclic policy is similar to calculating minimum mean cycles on graphs, and efficient algorithms of calculating such minima were obtained in several works (e.g., \citet{Karp78}). 
However, we stay in the context of optimizing long-run average revenue and propose our own algorithm for detecting the optimal cyclic policies that essentially has the minimal possible complexity i.e., in line with the literature on minimum mean cycles on graphs.

We first introduce a general class of revenue functions 
that we call weakly coupled revenue functions in which the revenue per period depends on a finite number of neighboring prices. 


\subsection{Definition of Weakly Coupled Functions}



For a pricing strategy $\pi \in \cP$, we define a class of  long-run average revenue functions called weakly coupled.

\begin{dfn}\label{def: weakly-depndt} {\bf (Weakly coupled revenue functions)}  We say that a function $\cR: \cP \to \R_+$ is a $\tau$-weakly coupled revenue function for some positive $\tau$, if there exists a function $f : \Omega^{\tau} \times \Omega^{\tau}  \to \R^+$, such that for any policy of the form, $\pi=\big((w_n:n\geq 1):w_n\in \Omega^{\tau}\big)$, the following holds:
\begin{align*} \cR(\pi)&= \frac{1}{\tau} \varphi(\pi),\\
\varphi(\pi)&:= \lim\sup_{n\rightarrow\infty} \frac{1}{n}\sum_{i=1}^nf(w_i,w_{i+1}).\end{align*}


\end{dfn}

\subsection{The Weakly Coupled Property of Revenue Functions in (BP) Models} \label{sec: Weakly Coupled_Changing Val}
We show next that the  revenue functions that appear in the context of (BP) models  introduced in Section~\ref{sec: Model}  are weakly coupled.



\begin{prop}\label{propweakly}
Under a (BP) model with patience level $\tau$, 
the revenue function is $\tau$-weakly coupled.
\end{prop}

\begin{proof} WLOG we can restrict ourselves to valuation processes that take values in $\Omega$. Let us fix a pricing policy that we write as $\pi=\big((w_j:j\geq 1): w_n\in\Omega^\tau\big)$.
The prices given by the vector $w_j$ are set during the $j^{th}$ time phase $[(j-1)\tau+1,j\tau]$. Given a realization of the processes $(\xi^i_t,\cV^i_t)_{i\in\cI}$, $t\geq 0$, let us denote by
$\hat{L}_j$ the revenues generated during the phase $w_j$. For $j \geq 2$, the revenues  $\hat{L}_j$ is the result of purchases that occur from either customers who arrive to the system during the $j^{th}$ phase and who buy during this same phase, or from customers who arrived to the system earlier and who are still in the system because they did not buy and did not reach the absorbing exit state, $v_0$. Because the patience level is bounded by $\tau$, the latter customers all come from the ${j-1}^{th}$ phase (during which prices $w_{j-1}$ are set). Hence $\hat L_j$ is a function of $(w_{j-1},w_j)$ and the realizations of $(\xi^i_t,\cV^i_t)_{i\in\cI}$, for time $t\in[(j-2)\tau+1,j\tau]$. 

Since the arrival process $(\xi^i_t,\cV^i_t)_{i\in\cI}$ is assumed to be stationary in $t$, it follows that the expectation $L_j$ of $\hat L_j$ over all  realizations of   $(\xi^i_t,\cV^i_t)_{i\in\cI}$ is a deterministic function $f(w_{j-1},w_j)$.

It then follows that, $$\mathcal{R}(\pi)=\limsup_{n\rightarrow\infty}\frac{1}{n\tau}\,\sum_{j=1}^n f(w_j,w_{j-1})$$
is a $\tau$ weakly coupled function of the pricing policies. $\square$
\end{proof}

\subsection{Optimizing Weakly Coupled Revenue Functions }

We consider a $\tau$-weakly coupled revenue function following Definition~\ref{def: weakly-depndt}. We are looking to find the policy that maximizes $\varphi$.
 We first extend the definition of $\varphi$ to finite strings $W=(w_1,...,w_n)$ as follows:
\begin{equation}\label{eq: phi}\varphi(W)=\frac{1}{n} \sum_{i=1}^{n} f(w_i,w_{i+1}),\end{equation} with the notation $w_{n+1}=w_1.$ Note that $\varphi(W)=\varphi(\pi)$, where $\pi=(W,W,...).$ When $W=(w_1,\ldots,w_P)$ is the shortest period of $\pi$ we write $\pi=(w_1,\ldots,w_P)$.

\begin{prop}\label{prop: Weakly Coupled General} { (General optimization).}
There exists a  cyclic $\pi^* \in \cP$ such that
$$\varphi(\pi^*)=\sup_{\pi \in \cP} \varphi(\pi)$$
Moreover,  $\pi^*=(w_1,w_2,\ldots,w_P)$, $w_i \in \Omega^{\bTau}$ and all the $w_i$'s are distinct, with $P \leq K^{\bTau}$.
\end{prop}

\bigskip

\begin{proof}  Denote by $\cW_n$ the set of strings of length less than $n$, that is $W \in \cW_n$ if $W=(\om_1,\ldots,\om_l)$ with $\om_i\in\Omega^{\bTau}$ for all $i\leq l$ and $l\leq n$. We have that $\varphi(W)=\varphi(\pi(W))$, where $\pi(W)=(W,W,...).$   Let $\bar W_n= {\rm argmax}_{W \in \cW_n} \varphi(W)$. From the definition of $\varphi$, it follows that for any $\pi \in \cP$,
$$\varphi(\pi) \leq   \varphi(\bar W_n)+ \frac{\max f}{n}.$$
Hence, it is sufficient to prove that for any fixed $n$,  the maximizer $\bar W_n$ can be taken to be simple. Suppose $\om$ is such that $\bar W_n=(W,W')$ with $W$ and $W'$ starting with some $\om \in \Omega^{\bTau}$, and let $N$ and $N'$ be the sizes of $W$ and $W'$. Then since
$$\varphi(\bar W_n)=\frac{1}{N+N'}(N \varphi(W)+N' \varphi(W'))$$
we get that $\varphi(W)=\varphi(W')=\varphi(\bar W_n)$. Continuing this procedure we reach a simple maximizer of $\varphi$ on $\cW_n$.  From that we also conclude that $$\varphi(\pi) \leq   \varphi(\bar W),$$ where $\bar W$ is the maximizer of $\varphi$ on all cyclic policies where the elements of the cycle are all distinct.

\hfill $\square$

\end{proof}


\vspace{0.2cm}

Since we are also interested in the constrained optimization on $\cA_\sigma$, we state the following proposition, which is immediately obtained from the proof of Proposition~\ref{prop: Weakly Coupled General} by restricting the pricing policies to those in $\cA_\sigma$.


\begin{prop} \label{cor: Weakly Coupled}  (Optimization on $\cA_\sigma$). For any $\bTau \geq 1$, for any function $f : \Omega^{\bTau} \times \Omega^{\bTau} \to \R^+$ we have the following.
There exists an $M$-simple  cyclic $\pi^* \in \cA_\sigma$ such that
$$\varphi(\pi^*)=\sup_{\pi \in \cA_\sigma} \varphi(\pi).$$
\end{prop}

By recalling Proposition~\ref{propweakly},  the previous proposition completes the proof of Theorem \ref{thm: Main}.

We observe here that the proof of the optimality of cyclic policies does not rely on the notion of regenerative points (compared to  \cite{Besbes15} or \cite{LiuCooper15}), but takes full advantage (see, proof of Proposition~\ref{prop: Weakly Coupled General}) of the defining expression for weakly coupled functions.

\subsection{Optimization Algorithm}\label{sec: Algo}

The objective of this section is to present an algorithm to find an optimal  cyclic policy $\pi^*$ as in Proposition \ref{prop: Weakly Coupled General}.

The algorithm is the result of a greedy-type construction that generates an optimal solution. We first fix two consecutive strings of prices $w_1$ and $w_2$, and set the third one so as to insure that the second string $w_2$ is the maximizer of the revenue generated by these three strings in the first two periods: $f(w_1,w_2)+f(w_2,w_3)$. Once $w_3$ (not necessarily unique) is identified, we move to find $w_4$, given $w_2$ and $w_3$. The complexity of the algorithm is reduced in particular thanks to the ``simple" property of the pricing policy.

Consider a finite string $W=(w_1,...,w_n$), we denote by
$$\tilde{\varphi}(W)= \frac{1}{n-1} \sum_{i=1}^{n-1} f(w_i,w_{i+1}).$$

\noindent We define the following map from $\Omega^\bTau \times \Omega^\bTau$ to the subsets of $\Omega^\bTau$,
$$\psi(w,w') = S; \quad S=\{ \bar{w} \in \Omega^\bTau : \varphi(w,w',\bar{w}) = \max_{\tilde{w} \in \Omega^\bTau} \varphi(w,\tilde{w},\bar{w})\}.$$
A collection $\cW_n$ of finite simple strings of the same size $n$ is called {\it admissible} if for all $W \in \cW_n$, it is of the form $(w_1,..,w_n)$ with the same $w_1:=a$; moreover, if  $(W,W')\in {\cW_n}^2$ such that $w_n=w'_n$, then $W=W'$.
Set $\tilde{K}=K^\bTau$.

\medskip

\noindent {\bf Algorithm.}
{\sf \begin{itemize}
\item[{\sc Step 1.}] {\sc Initialization.} Fix $(w_1,w_2)\in \Omega^\bTau\times\Omega^\bTau$, with $w_1\neq w_2$, let $\cW_2(w_1,w_2)= \{(w_1,w_2)\}$, $\cS_2(w_1,w_2)= \emptyset$ and set $n=2$
\item[{\sc Step 2.}] {\sc Computation.} Given an admissible simple collection $\cW_n$ with $n<\tilde{K}$ and $\#\cW_n\leq \tilde{K}$, define an admissible collection $\cW_{n+1}=\Psi(\cW_n)$ in the following way
\begin{itemize}
\item[1.] for each $W^i \in \cW_n :=(w_1^i,\ldots,w_n^i)$, consider $S^i=\psi(w_{n-1}^i,w_n^i):=\{s_1^i,\ldots,s_{J_i}^i\}$  where, $J_i\leq \tilde{K}$ depends on the pair $(w_{n-1}^i,w_n^i)$. Consider all the strings $$\big\{W^i_j= (w_1^i,\ldots,w_n^i,s^i_j): 1\leq j \leq J_i, ~1\leq i\leq \#\cW_n\big\}.$$
\item[2.] Select from these strings those such that $s^i_j=w_1$. This set of strings (possibly empty) is denoted by $\cS_{n+1}$
\item[3.] Eliminate from the remaining all the strings $W^i_j$'s that are not simple
\item[4.] From those remaining, consider any two strings, such that $s^i_j=s^{i'}_{j'}$ for some $(i,j)$ and $(i',j')$. If $\tilde\varphi(W^i_j)<\tilde\varphi(W^{i'}_{j'})$ or  $\tilde{\varphi}(W^i_j)>\tilde{\varphi}(W^{i'}_{j'})$, then eliminate the string that has the smaller value of $\tilde{\varphi}$. If $\tilde\varphi(W^i_j)=\tilde\varphi(W^{i'}_{j'})$, we eliminate (randomly) one of them. We do that for all such strings. The remaining strings must define an admissible collection of simple strings that we denote by $\cW_{n+1}$ and note that the cardinality $\# \cW_{n+1} \leq \tilde{K} $
    \end{itemize}
\item[{\sc Step 3.}] {\sc Iteration.} If $n=\tilde{K}$, STOP. Otherwise, set $n=n+1$ and GO TO Step 2.
\end{itemize}}

We then make the following elementary observation.
\begin{lem} \label{prop: algo} If the  cycle of an optimal strategy of  Proposition~\ref{prop: Weakly Coupled General} is ${W}:=(w_1,w_2,\ldots,w_P)$, for some integer $P\geq 2$,  then $W \in \cS_P(w_1,w_2)$.
\end{lem}

Given that each computation step generates simple strings, and by noticing that the first part of Step 2 can be done offline, we need at most $\cO(\tilde{K}^2)$ computations to obtain $\cW_{n}(w_1,w_2)$ and $\cS_{n}(w_1,w_2)$, $n\leq \tilde{K}$. This algorithm takes as an input, $w_1$ and $w_2$ and hence, needs to be repeated for all possible values of $w_1$ and $w_2$. As a consequence of Lemma \ref{prop: algo}, if we consider
$$\cS:=\bigcup_{(w_1,w_2), w_1\neq w_2} \bigcup_{n\leq \tilde{K}} \cS_n(w_1,w_2)$$
and the set of single value strategies $$\cS_0:=\bigcup_{w \in \Omega^\bTau} \{(w)\}$$
then an optimal strategy $W$ of  Proposition \ref{prop: Weakly Coupled General}  satisfies $W \in \cS$. Note that the computation of $\cS$ needs at most $\cO(\tilde{K}^4)$ calculations of the type $\varphi(w,w',w'')$.  We have thus proved the following.

\begin{prop} \label{propo: algo} If a revenue function  is $\bTau$-weakly coupled, then an optimal cyclic policy of Proposition \ref{prop: Weakly Coupled General} can be obtained with  $\cO({K}^{4\bTau})$ elementary calculations \footnote{Observe that in step 2, when finding all admissible  strings that start with $(\omega_1,\omega_2)$ we get actually the information on how other couples such as $(\omega_2,\omega_3^1)$ can be completed in admissible strings. If one takes into account this information, the algorithm here presented can be slightly altered to give a complexity of the order of $\tilde{K}^3$ which corresponds to the minimal complexity known for equivalent problems such as minimum mean cycles on graphs.}.
\end{prop}

As for the constrained optimization, the immediate corollary of Proposition \ref{propo: algo} is the following.

\begin{cor} \label{cor: algo} If a revenue function  is $\bTau$-weakly coupled, then any optimal $M$-simple cyclic $\sigma$-policy of Proposition \ref{cor: Weakly Coupled} can be obtained with  $\cO({K}^{4M})$ elementary calculations.
\end{cor}

Given Proposition~\ref{propweakly}, the previous corollary completes the proof of Theorem \ref{thm: Algo}.

\section{Markovian Models of Time-Varying Valuations}\label{sec: Markovian_Model}

In this section, we revisit the model introduced in Section~\ref{sec: simple}, where we specifically restrict ourselves to the class of (BP) valuation processes that follow a Markov chain. This special and important class allows us to achieve a number of things. First, we obtain an explicit expression for the revenue function. This is particularly helpful in the numerical study of Section~\ref{sec: Numerics}. Besides retaining the weakly coupled property, the revenue function  is shown to be affine in the duration of a phase and in the expected number of customers available initially. As a result, we obtain some refinements of the main results of Section~\ref{sec: Main results}. 
Finally, this Markovian setting allows us also to extend our analysis to the case where  the patience levels are bounded only in expectation   (see, Section~\ref{sec: extensions}).

\subsection{Model Ingredients}

Similar to Section~\ref{sec: simple}, and to simplify the presentation, we consider a constant number of customers arriving every period. Without loss of generality, we assume that this number is one. We restrict our analysis to a single class of customers since the revenue function of a multiple class model is a linear combination of revenue functions of single class models, and because the additional properties that we want to show for the revenue function of the Markovian model are stable under linear combination.

We recall that every arrival is endowed with an initial valuation $\omega \in \Omega\equiv\{v_{j}:1\leq j\leq K\}$ with a distribution given by a probability vector $\boldsymbol\gamma=(\gamma_1,...,\gamma_K)$.

We also assume that the valuation process of the customer arriving at time $t$, $(\cV_t(s):s\geq 0)$ (where we dropped the superscripts $i$ and $k$), follows a Markov chain on $\Omega$, determined  by his arrival valuation and by a transition matrix that we denote by $\bar{Q}$.
Each entry $q_{ij}$ of $\bar Q$, is the probability that a customer with current valuation $v_i$ sees his valuation change to $v_j$ in the next period. Recall that $v_0:=0$, is an absorbing state that once reached indicates that the client lost interest in purchasing. Hence, we take the first line of $\bar Q$ to be zero except for its first entry that is equal to $1$.  From now on, we refer to the transition matrix by its minor $Q$ obtained by removing from $\bar Q$ the first line and the first column.
This model allows customers to have different life spans, but  we recall that there exists a maximum patience level $\bbtau$ so that $\cV_t(s)=v_0=0$ for every $s\geq \bbtau$. 
\vspace{0.2cm}

In this Markovian setting, besides $\bbtau$, the primitives of the problem are summarized by $(\boldsymbol\gamma,Q)$. With some abuse of notations, the pair $(\boldsymbol\gamma,Q)$ will also be used to represent the process $\cV_t$. 

When $\bbtau<\infty$, we call our Markovian model (M-BP), for {\it Markovian, Bounded Patience}. The (M-BP) model satisfies  the assumptions of the general setting discussed in Section~\ref{sec: Model_Main}, that is: it is a (BP) model. Consequently, its revenue function is weakly coupled and all of the results obtained for (BP) models hold for these Markovian models.

In contrast,  we refer to the case  where $\bbtau=\infty$, as the  (M-UP) model for {\it Markovian, Unbounded Patience}. Section~\ref{sec: extensions} covers how the analysis of the (M-BP) yields approximate solutions to the  (M-UP) case.

\subsection{Revenue  Pairs} \label{sec: revenue function  pair}

 We start this section by defining a revenue pair that will be helpful in calculating the revenue function associated to the (M-BP) model.  

We say that a vector $\boldsymbol\theta$ denotes the expected number of customers in the system at a given time, when the coordinates of $\boldsymbol\theta$,   $\theta_m$, $1\leq m\leq K$ are equal to the expected number of customers in the system, at this time, with valuation $v_m$. Note that the number of customers in the system at any point in time is bounded by $\bbtau$.

Let   $\mathcal{M}= \Omega \times\mathbb{N}\times[0,\bbtau]^K$. For the rest, we denote by $t$ the arrival time of a customer, and  by $\tauu$ the duration of a generic phase of a pricing policy.

  \begin{dfn}[Revenue pair] \label{defTheta}  Let $\boldsymbol\theta$ be a vector determining the expected number of customers in the system at the beginning of a phase $(k,\tauu)$, we define
  \begin{align*} \boldsymbol\Theta&:  \hspace{0.85cm}   \mathcal M \rightarrow[0,\bbtau]^K\\
 & \quad (k,\tauu,\boldsymbol\theta) \mapsto \boldsymbol\theta'=\boldsymbol\Theta\big(k,\tauu|\,\boldsymbol\theta \big),\end{align*}
where $\boldsymbol\theta'$ is the vector  measuring the {expected} number of customers in the system at the end of the phase $(k,\tauu)$.

Let
  \begin{align*} L&:  \hspace{0.85cm}   \mathcal M \rightarrow\R_+\\
 & \quad (k,\tauu,\boldsymbol\theta) \mapsto L\big(k,\tauu|\,\boldsymbol\theta \big),\end{align*}
where $L\big(k,\tauu|\,\boldsymbol\theta \big)$ is the total expected revenues generated during phase $(k,\tauu)$.
 \end{dfn}

 It is worth stressing that $\boldsymbol\theta'$ accounts (in expected value) for customers
 included in $\boldsymbol\theta$ that did not exit during the phase $(k,\tauu)$, and those that arrived during the phase $(k,\tauu)$ and did not yet exit.  Similarly, $L\big(k,\tauu|\,\boldsymbol\theta \big)$ are the expected revenues generated from customers included in $\boldsymbol\theta$ or from those that arrived to the system during the phase $(k,\tauu)$. Regardless, these revenues were generated during time $\tauu$.

Considering that the {expected} number of customers in the system, with valuation $v_m$, at the end of the phase $(k,\tauu)$, depends only on $(k,\tauu)$ and  $\boldsymbol\theta$, is a consequence of : $i.)$ the Markovian feature of the valuation process that guarantees that the evolution of the customers included in $\boldsymbol\theta$ does not depend on the past, and $ii.)$ the stationarity of the processes $(\boldsymbol\gamma,Q)$, (i.e., being independent of $t$). For these reasons we can also define the expected revenue as a function of $(k,\tauu)$ and  $\boldsymbol\theta$.

We now express the revenue function {\it via} the pair $(\boldsymbol\Theta,L)$. Since we are interested in the optimization of the  long-run average revenue, we may assume without loss of generality that the system starts empty at $t=0$.


\begin{prop}\label{Eq: Main Objective} Let $\pi=\big((k_j,\tauu_j):j\geq 1\big)$. The revenue function is given by
\begin{align*}\mathcal{R}(\pi)&=\limsup_{n\rightarrow\infty}\frac{1}{T_n}\,\sum_{j=1}^n L(k_j,\tauu_j|\,\boldsymbol\theta^{j-1}),\\
 \boldsymbol{\theta}^{j}&=\boldsymbol\Theta\big(k_j,\tauu_j|\,\boldsymbol{\theta}^{j-1}\big).\end{align*}
with $T_n=\sum_{i=1}^n{\tauu_i}$.\vspace{0.2cm}
\end{prop}

\begin{proof} Fix a pricing policy $\pi=\big((k_j,\tauu_j):j\geq 1\big)$. For a realization through time of the arrival and the transition processes $(\cV_t(s):s\geq 0)$, we define $\hat{\boldsymbol{\theta}}^{j}$ as the realized number of customers  in the system at time $T_j$. As in  Definition \ref{defTheta} we have that the $\hat{\boldsymbol{\theta}}^{j}$ and the realized revenue during the $j^{th}$ phase are functions of $\hat{\boldsymbol{\theta}}^{j-1}$, as well as the realization of the arrival and the transition processes of customers during this phase. Considering all the possible realizations of $(\boldsymbol\gamma,Q)$ for $0\leq t\leq T_j$, we find that the expected value of $\hat{\boldsymbol{\theta}}^{j}$ is exactly the sequence $\boldsymbol{\theta}^{j}=\boldsymbol\Theta\big(k_j,\tauu_j|\,\boldsymbol{\theta}^{j-1}\big)$ and that the expected value of the revenue generated during the $j^{th}$ phase is $L(k_j,\tauu_j|\,\boldsymbol\theta^{j-1})$, from which Proposition \ref{Eq: Main Objective} immediately follows.  $\square$
\end{proof}

\subsection{The (M-BP) Model. Main results}



Our first result shows that under the (M-BP) model  the revenue function satisfies the so-called $\tau$-affine property. We use the notation $\langle \cdot,\cdot\rangle$ to denote the scalar product of two vectors.

\begin{prop}\label{prop:Eps_Affine}{\bf (Affine property)}
Under (M-BP) and given  $(\boldsymbol\gamma,Q,\tau)$, the following holds. For any price $1\leq k \leq K$, there exist $ \bar{\boldsymbol\Theta}_k\in [0,\tau]^K$, $A_k, B_k \in \R^+$, 
and $\bold{C}_k\in \R^K$ such that
for any $\tauu \geq \tau$, and for any ${\boldsymbol\theta}\in [0,\tau]^K$, we have that
\begin{itemize}
\item[(A1)]  $\boldsymbol\Theta\big(k,\tauu|\,\boldsymbol{\theta}\big)= \bar{\boldsymbol\Theta}_k,$
\item[(A2)]   $  L(k,\tauu|\,{\boldsymbol\theta})= A_{k}+B_{k}(\tauu-{\tau}) +\langle { \bold{C}_k,\boldsymbol\theta}\rangle.$
\end{itemize}

In this case, the pair  $(\boldsymbol\Theta,L)$ is said to be ${\tau}$-affine. The parameters $A_{k}, B_{k},\bold{C}_{k}$ and $\bar{\boldsymbol\Theta}_k$ can be given in closed form.
\end{prop}

The proof of the affine property is postponed to Appendix A.

This previous property not only allows one to obtain an explicit expression of the revenue function but also leads to refinements of the main results of Section~\ref{sec: Main results}.

While $\cA_\sigma$ is the set of policies where each phase is a multiple of $\sigma$, we introduce $\cP^+_\sigma$ as the set of policies where each phase has a duration larger than $\sigma$. Clearly, $\cA_\sigma\subseteq\cP^+_\sigma$.




\begin{prop}\label{prop: Affine_Reduction} Under (M-BP):
\begin{itemize}
\item[i.)] One can find an optimizing strategy of (\ref{Eq: General_Main_Form}) that is either a fixed-price strategy or a strategy in which no price is set for more than $\tau$ periods.
\item[ii.)]If $\sigma\geq \tau$, there exists a simple cyclic
  $\sigma$-policy that optimizes $\cR$ on  $\mathcal{P}^{+}_{\sigma}$. Moreover, there exists a threshold $\sigma_0$ such that, if $\sigma\geq\sigma_0$, then a fixed-price policy is optimal on $\cP^+_{\sigma}$.
\end{itemize}

\end{prop}

\begin{proof}
\begin{itemize}
\item[$i.)$]
Suppose that the optimum cannot be reached by a fixed-price policy. Let $\pi^*=((k_1,\tauu_1),\ldots,(k_n:\tauu_n))$ be a cyclic optimal strategy. We will show by contradiction that $\tauu_1\leq \tau$. Suppose that $\tauu_1 > \tau$. Replace $\pi^*$ by  $\pi_s=((k_1,s),(k_2,\tauu_2),\ldots,(k_n,\tauu_n))$ with  $s \in [\tau,\infty)$. Observe that
 $$\cR(\pi_s)= \frac{1}{s+\sum_{i=2}^n  \tau_i} (U+V(s-\tau)),$$
where $U$ and $V$ do not depend on the value of $s \geq  \tau$. It follows that the derivative of $\cR(\pi_s)$ has a sign that does not depend on $s$. The sign must be negative otherwise the optimum could be reached by the fixed-price policy with price $k_1$.
Hence  $\cR(\pi_s)$ for $s\in  [\tau,\infty)$ reaches its optimum at $s=\tau$ a contradiction  with $\tauu_1 > \tau$. Because the choice of which phase is first in a cyclic strategy is arbitrary, hence, all the $\tauu_l$ are less than $\tau$.

\item[$ii.)$]
The same argument as the one used in the proof of Proposition \ref{prop: Affine_Reduction} $i.)$, implies that for $\sigma \geq \tau$ the optimum of  $\cR$ on  $\mathcal{P}^{+}_{\sigma}$ is reached on a fixed-price policy or on  a $\sigma$-policy  $\pi=((k_j,\tauu_j):j\geq 1)$ with  $\tauu_j =\sigma$ for all $j \in \N$. By Proposition \ref{cor: Weakly Coupled}, an optimal   $\sigma$-policy  can be considered simple cyclic (in particular $k_{j+1}\neq k_j$ unless it is a fixed-price policy).  Let $\pi_c =(k_1,\ldots,k_n)$ (where each $k_i$ is set for $\sigma$ periods) be an optimal policy.

From the definition of affine revenue functions, we note  that
$$\cR(\pi_c)=  \frac{1}{n \sigma} \sum_{i=1}^{n}  A_i+B_i(\sigma-\tau)+T_{i,i+1}$$
where $A_i$ and $B_i$ are functions of $k_i$ and $T_{i,i+1}=\langle { \bold{C}_{i},\boldsymbol{\bar{\Theta}}(i+1)}\rangle$ is a function of $(k_i,k_{i+1})$ (and not of  $\sigma$) and where $k_{n+1}:=k_1$.  Observe also that $B_i$ is the long-run average revenue of the fixed-price policy with price $k_i$.

As $\sigma \to \infty$, it is clear that the fixed-price policy with price $k$ such that $B_k = \max_{k'} B_{k'}$ becomes an increasingly better approximation of an optimal policy in   ${\mathcal A}_{\sigma}$. In the generic case where $\max_{k'} B_{k'}$ is a strict maximum for $k'=k$, then as soon as $\sigma\geq \sigma_0$, where $\sigma_0$ is such that for any simple cyclic vector $(k_1,\ldots,k_n)$, we have
$$B_k(n\sigma_0 + \tau - \sum_{i=1}^{n} \frac{B_i}{B_k} (\sigma_0-\tau)) >  \sum_{i=1}^{n}  A_i+T_{i,i+1},$$
then the fixed-price policy with price $k$ is strictly optimal on $\cP^+_{\sigma}$.
\hfill $\square$
\end{itemize}
\end{proof}\bigskip




The previous proposition brings some refinements to Theorem~\ref{thm: Main}.  In the unconstrained case, since an optimal strategy is $\tau$-simple and cyclic, we already know that no single price is set during more than $2\tau$ periods. However, the affine structure of the revenue reduces this bound to $\tau$ periods. Moreover,  recall that Theorem~\ref{thm: Main} characterized an optimal policy in $\cA_\sigma$. When $\tau\geq \sigma$  part $ii.)$ of the previous proposition generalizes the statement of Theorem~\ref{thm: Main} to $\cP_{\sigma}^+$. Finally, note that $ii.)$ of this proposition claims also that if the pricing pace is too slow ($\sigma$ large), then a fixed-price policy is optimal.

\section{Absence of Structure and Complexity}\label{sec: Complexity}
We devote this section to complementing our discussion on the complexity of the solutions obtained in Theorem~\ref{thm: Main} and Proposition~\ref{prop: Weakly Coupled General}. We provide two examples, respectively; one for the general setting of weakly coupled revenue function and one for the specific model of changing valuation, which show  unexpected optimal policy behavior. These examples confirm that one cannot expect to reduce the optimization (feasible) set through some simplifying structure of the optimal policy.

\subsection{Example of an Optimum Along an Arbitrary Cyclic Strategy for  Weakly Coupled Revenue Functions.}

We revisit the general setting of weakly coupled revenue functions. Despite the results obtained in Section~\ref{sec: Weakly Coupled}, the question of whether one can reduce the complexity of such cyclic policy even further remains. The next simple example shows that in general this is not possible. In particular, any cycle of any size can be an optimal solution of an adequately chosen  weakly coupled revenue function. \vspace{0.2cm}

\begin{exm}Let $\bar\pi=(w_1,...,w_n)$ be a finite simple cyclic strategy. Fix $U>0$, and define $f$ such that
$$f(w,w')= \begin{cases} U \text{ \ if \ } (w,w') = (w_i,w_{i+1}) \text{\ for some  } i \in [1,n] \\
0 \text{ \ if not, }  \end{cases}$$
where we used the notation $(w_n,w_{n+1}):=(w_n,w_{1})$.

\noindent The following is straightforward from the definition of $f$. We have
$$\varphi(\bar \pi)=U,$$
and for any strategy $\pi \neq \bar\pi$, we have that
$$\varphi(\bar\pi)-\varphi(\pi) \geq \frac{1}{K} U.$$
\end{exm}


\bigskip


\subsection{Example of an M-BP Model With an Increasing Optimal Cyclic Policy With Length $K-1$.} \label{Sec: Example}

In this section we construct the example that proves Proposition~\ref{prop: Example_basic_intro}. This proposition confirms that due to the changing valuations feature of the model, one can construct an optimal cyclic policy that is increasing and involves all but one  of the available prices.

The idea behind this example is the following. First, we standardize the return of any constant price strategy and set it equal to $1$ by carefully choosing the parameters $\boldsymbol\gamma$ and {\bf v}. Next we choose the transition probabilities $q_{i,j}$ as follows. We say that a customer is in state $i$ if her current valuation is $v_i$. We set $q_{i,j}=0$, from any state $i\in [1,K-2]$ to { any} state $j\leq K$, except for $j=i+2$. We denote by  $(i\, (i+1))$ the expected revenue generated by all customers with valuation strictly lower than $v_i$, who enter at a date where the price is $i$ and see the price $i+1$ at the subsequent date. We set $q_{i,i+2}$ so that $(i\, (i+1))$ is a positive quantity $U\ll 1$ that does not depend on $i\in [2,K-1]$. Since the transition probabilities from $i-1$ to states that are { strictly} higher than $i+1$ are zero, one observes that whenever a price $i$ appears in the strategy, it is { optimal} to follow it with the price $i+1$, otherwise, there is a shortfall of $U$ in the revenue.
Finally, the transition probability, ${q_{K-1,j}}$, is set to zero for all states $j$, except for state $j=2$; while the transition $q_{K-1,2}$ is chosen sufficiently large so that the expected revenue $(K\,2)$ is of order $\frac{11}{10}\,U$. The latter forces an optimal strategy to contain the string $(K,2)$. Then, as discussed above, $2$ must be followed by $3$, then $4$, etc., until  $K$, which determines the optimal strategy as  $(2,3,\ldots,K-1,K)$.


 The values of ${\bf v},\boldsymbol\gamma$ and $Q$,  are designed so that the key properties (E0)--(E2), defined below, hold. These properties are sufficient to prove the optimality of the strategy $\bar\pi=(2,\ldots,K)$. We first set ${\bf v}$ and $\boldsymbol\gamma$ such that (E0) and (E1) hold. Then the first $K-2$ rows of $Q$ are chosen so that the first identity of (E2) holds. Then the row $K-1$ is chosen to guarantee the second identity of (E2). Finally the last row of $Q$, that does not play a role in the optimization, is chosen to be $(0,0,\ldots,0,1)$.


\begin{exm}~\label{ex2}  Fix $K \geq 5$. Let $ \varepsilon=\frac{1}{200K}$. 
 Take ${\bf v},\boldsymbol\gamma$ and $Q$ as follows.
\begin{align*}
\forall  i \in [1,K] : \gamma_i&=(1-\varepsilon)\varepsilon^{-K+i} \\
\forall  i \in [1,K] :  v_i&=\varepsilon^{K-i} / (1-\varepsilon^{K+1-i}) \\
\forall  i \in [2,K-1], \forall j \neq i+1 :  q_{i-1,i+1}&=\frac{10}{11} \varepsilon^{K+1} (1-\varepsilon^{K-i}),  \quad   q_{i-1,j}=0  \\
 \forall j \neq 2 : q_{K-1,2}&=\varepsilon^2 (1-\varepsilon^{K-1})/(1-\varepsilon), \quad   q_{K-1,j}=0\\
 \forall j\in [1,K -1] : q_{K,j}&=0, \quad  q_{K,K}=1
\end{align*}

Let $\Gamma_i=\sum_{l=i}^K \gamma_l$. With our choice of parameters, let $\a=1$ and $U=(1-\varepsilon) \frac{10}{11} \varepsilon^{K-1} \leq \a$. Then the properties \eqref{e0} -- \eqref{e2} hold true, where we denote by $\bq_{i,j}:=\gamma_i q_{i,j}$, and where $\xi \in [\frac{1}{10},\frac{1}{10}+2\varepsilon]$.

\begin{equation}  \label{e0} \tag{E0} \forall i \in [1,K] : \Gamma_i v_i=\a \end{equation}

\begin{equation} \label{e1} \tag{E1} \forall  i \in [1,K-1]  : v_{i} \leq \frac{\varepsilon}{1-\varepsilon} v_{i+1}.\end{equation}

\begin{equation} \label{e2}\tag{E2}
\begin{cases}
\forall  i \in [2,K- 1]  : \bq_{i-1,i+1}v_{i+1}=U, \quad  \forall j \neq i+1 : q_{i-1,j}=0\\
(\bq_{1,3}+\bq_{2,4}+\ldots + \bq_{K-2,K} + \bq_{K-1,2}) v_2= (1+\xi)U,  \quad  \forall j \neq 2 : q_{K-1,j}=0\\
\forall  j\in [1,K -1] : q_{K,j}=0, \quad q_{K,K}=1\end{cases} \end{equation}
\end{exm}

The following claim shows that {\bf Example \ref{ex2}} satisfies Proposition \ref{prop: Example_basic_intro}.

\medskip

\noindent{\sc Claim.}    {\it Let {\bf $\bar\pi=(2,\ldots,K)$}. We have that $\cR(\bar\pi)=\a+U+ \frac{\xi}{K-1}U$ and for any $\pi \neq \bar\pi$
    $$\cR(\pi)\leq \cR(\bar\pi)-\frac{1}{20(K-1)}U.$$}

\begin{rem} Observe that for   \eqref{e0} -- \eqref{e2}  to hold, while $K\geq 5$, we must have $U=\cO(\varepsilon^{K-4} \a)$; thus the loss if one adopts a fixed-price strategy with any price $v_2,\ldots,v_K$, is small compared to an optimal strategy. Note also the importance of the term $\bq_{K-1,2} v_2$ in our assumptions since $(\bq_{1,3}+\bq_{2,4}+\ldots + \bq_{K-2,K}) v_2 \leq \varepsilon \,U$.

\end{rem}

\noindent{\sc Proof of the Claim.}\\
We denote by $(ij)$ the expected revenue function  generated by the customers with valuation strictly lower than $v_i$, who enter at a date where the price is $i$ and see the price $j$ at the subsequent date. Thus,  for $\pi=(w_1,\ldots,w_J)$
$$\cR(\pi)=\a+\frac{1}{J} \left(\sum_{i=1}^{J-1} (w_i w_{i+1})+(w_J w_1)\right).$$

\begin{lem} Under (E1)--(E2) we have that
\begin{align*}
(i (i+1))&=U  : \forall i \in [1,K-1] \\
(i (i+\ell))&=0 : \forall 1 \leq i \leq  i+\ell \leq K, \ell\geq 2 \\
(K 2)&=\frac{11}{10} U  \\
(i j)&\leq \left(1+\frac{1}{100 K}\right)  U  :  \forall i,j \in [1,K], (i j) \neq (K 2). \end{align*}
\end{lem}
\begin{proof}
The first three equalities follow directly from the assumption (E2).

Consider $i \leq K-1$ and $2 \leq j \leq i$. We have that
$$(i j)= (\bq_{j-2, j}+\bq_{j-1, j+1}+\ldots+\bq_{i-1,i+1})  v_j \leq (1+2 \varepsilon)U$$
due to assumptions (E1) and (E2) (when $j=2$ we used the notation $\bq_{0,2}=0$).

Now, for $j >2$
$$(K j)= (\bq_{j-2, j}+\bq_{j-1, j+1}+\ldots+\bq_{K-2,K})  v_j \leq (1+2 \varepsilon)U.$$
and finally,
$$(K 1)= (\bq_{1, 3}+\bq_{2, 4}+\ldots+\bq_{K-2,K})  v_1 + \bq_{K-1,2}v_1 \leq  \varepsilon \frac{11}{10}U.$$ \hfill 
\end{proof}

It follows from the first two equalities of the lemma that  $\cR(\bar\pi)=\a+U+\frac{1}{10(K-1)}U$.
Take $\pi$ to be another simple cyclic strategy. If $\pi$ does not contain the string $(K,2)$ it follows from the third equation of the lemma that
  $\cR(\pi)\leq \a+ \left(1+\frac{1}{100 K}\right)U \leq \cR(\bar\pi)-\frac{1}{20(K-1)}U.$  If, to the contrary, $\pi$ contains the string $(K,2)$ but is distinct from $\bar \pi$, then
  $\pi$ contains another string $(i,j)$ with $j\neq i+1$; hence using the first, second, and third inequalities of the lemma we get if $L$ is the length of the cycle of $\pi$ :
  \begin{align*}\cR(\pi)&\leq \a+\frac{1}{L}\left(\frac{11}{10} U +(L-2)\left(1+\frac{1}{100 K}\right) U\right) \leq \cR(\bar \pi)-\frac{1}{20(K-1)}U \end{align*}
  and the claim is proved.  \hfill $\square$

\subsection{Approximating the Solution of Problem~(\ref{Eq: General_Main_Form})}

The lack of simplifying structure of the  optimal policies confirms the curse of dimensionality, which seems inherent to the changing valuations feature of such type of intertemporal pricing.
This was argued in Section~\ref{sec: simple} by relying specifically on example~\ref{ex2}. In such context, the exponential complexity makes our results even more valuable in suggesting efficient approximations for problem (\ref{Eq: General_Main_Form}). Indeed, as a result of the untractability of problem (\ref{Eq: General_Main_Form}), one would naturally look for approximations. For instance, one could consider optimizing over fixed-price policies. Such policies are shown to be asymptotically optimal in various settings (e.g., the recent work of \citet{Yiwei_Far_Trich18} who prove the asymptotic optimality of fixed-price policies under price commitment). Another approach is to approximate a customer's stochastic valuation by a carefully selected constant and accordingly, use the results of \citet{LiuCooper15}. Our results can be viewed as a third approach. By introducing the notion of pricing pace, we formulated (\ref{Eq: General_Pace_Form}) which is a constrained version of (\ref{Eq: General_Main_Form}) (and of independent interest). Hence, the optimal solutions of (\ref{Eq: General_Pace_Form}) defined and obtained through Theorem~\ref{thm: Main} and \ref{thm: Algo} are natural approximations of (\ref{Eq: General_Main_Form}). Moreover, such approximations by construction preserve the nature of the problem with respect to changing valuations. We prove by many numerical examples (see, Section~\ref{sec: Numerics}, Finding III and V), that such solutions even for small values of $M$ outperform the two other approaches of fixed-price policies and constant valuations. This in turns shows that the algorithm does give an added value to the pricing optimization even if it is operated in a constrained setting.

\medskip

\section{Extensions}\label{sec: extensions}

\subsection{Unbounded Patience Setting. The (M-UP) Model}

One main assumption  throughout this paper is the boundedness of the maximum patience level.  In this section, we consider relaxing this assumption and setting $\bbtau=\infty$ in the Markovian model. This model was denoted by (M-UP) and is only determined by the pair $(\boldsymbol\gamma,Q)$. Under this setting, the time at which the valuation process eventually reaches the value $v_0=0$ is endogenous, governed only by the matrix $Q$. We still assume that its expected value is finite. A sufficient condition for this to hold is
\begin{equation} \|Q\|:=1-\nu<1, \tag{$\mathcal C$} \label{gap} \end{equation}
for some $\nu\in(0,1)$, where $ \|\cdot\|$ is the infinity norm for matrices, that is the maximum absolute row sum of the matrix. This same condition guarantees that there are no accumulation of customers in the system, equivalently, at any time $t$ the expected number of customers in the system is uniformly bounded (see, Appendix~A for a proof of this claim).

We now analyze the (M-UP) case by approximating it by an (M-BP) model. For this,  fix any $\varepsilon>0$ and introduce the quantity \begin{equation}\label{Eq: tau_bar_eps} {\tau}_\varepsilon=\left|\ln\varepsilon/\ln (1-\nu)\right|.\end{equation}

By definition,  $1-\nu$ is an upper bound on the probability of a customer - no matter the state - to remain for an additional period in the system. It then follows from \eqref{gap} and the expression for $\tau_\varepsilon$ that the probability that an arriving customer  spends more than $\tau_\varepsilon$ in the system is lower than $(1-\nu)^{\tau_\varepsilon}=\varepsilon$. Therefore, we introduce the (M-BP) model with patience level $\tau_\varepsilon$ and the same primitives $(\boldsymbol\gamma,Q)$ as the (M-UP)
we are studying.

To proceed, we say that a pricing policy $\pi^*$ is $\epsilon$-optimal on a subset $\cA$ of $\cP$ for some $\epsilon>0$, if \[\sup_{\pi\in\cA}\mathcal{R}(\pi)\leq \mathcal{R}({\pi^*})+\varepsilon.\]

We have the following

\begin{prop} Under condition \eqref{gap}, if
$\pi^*$ is an optimal strategy on any subset $\cA$ of $\cP$, for the (M-BP) model given by $\tau_\varepsilon$ and $(\boldsymbol\gamma,Q)$, then $\pi^*$ is  $(v_K \varepsilon/\nu)$-optimal on $\cA$ for the (U-BP) model given by $(\boldsymbol\gamma,Q)$.
\end{prop}

\begin{proof}
Consider one customer that arrives to the system at some time $t$ with a valuation $v_m$ governed by a (U-BP) model to which we apply some policy $\pi\in\cA$. Recall that an (M-BP) model disregards any revenue generated by any customer that spends more than $\tau_\varepsilon$ in the system. Therefore, the difference between the expected revenues generated from respectively applying the same policy $\pi$ to both an (M-UP) model and an (M-BP) model with the same primitives, is bounded by $\eta^{\pi}(m,\tau_\varepsilon)\,v_K$, where $\eta^{\pi}(m,\tau_\varepsilon)$ is the probability in the (M-UP) model, that under $\pi$, the customer remains in the system  for a period greater than $\tau_\varepsilon$. By definition of $\tau_\varepsilon$ and condition~\eqref{gap}, $\eta^{\pi}(m,\tau_\varepsilon)\leq \varepsilon/\nu$. By considering all customers that arrived to the system by time $t$, the difference in expected revenues from these $t$ customers throughout the horizon will be upper bounded by $t\,v_K\varepsilon/\nu$. Let us denote by $\cR^U(\pi)$ (resp., $\cR^B(\pi)$) the long-run average revenue generated by applying policy $\pi$ under an (M-BP) (resp., (M-UP)) model. From the above we conclude that for any policy $\pi\in\cA$ we have that
$$\cR^U(\pi)\leq\cR^B(\pi)+v_K\,\varepsilon/\nu\leq \cR^B(\pi^*)+v_K\,\varepsilon/\nu\leq \cR^U(\pi^*)+v_K\,\varepsilon/\nu.$$ This completes our proof.
$\square$
\end{proof}

\medskip

As a corollary we obtain the following

\begin{prop}\label{optimal_results}
Under (M-UP), Theorems~\ref{thm: Main} and \ref{thm: Algo} as well as Propositions~\ref{prop:Eps_Affine} and \ref{prop: Affine_Reduction} hold after replacing the optimality concept by $\varepsilon$-optimality, and $\tau$ by $\tau_\varepsilon$.
\end{prop}

\subsection{Infinite Horizon Discounted Revenue Function } \label{infinite_horizon}
In this extension, we address how the main results can be easily extended to the context of infinite-horizon discounted revenue functions. For this, we can revisit the weakly coupled framework of Section~\ref{sec: Weakly Coupled} and extend  Proposition~\ref{prop: Weakly Coupled General} accordingly. The rest of the analysis on weakly coupled functions and its implications on problems~(\ref{Eq: General_Main_Form}) and (\ref{Eq: General_Pace_Form}) similarly extend.

\noindent We start by a definition of weakly coupled revenue function in the discounted context. Let
$$\cR(\pi)=\sum_{n=0}^\infty e^{-r n} f(w_n,w_{n+1}).$$ The definition of $\varphi$ given in Equation~(\ref{eq: phi}) is also modified in a similar way. Finally, we say that a policy $\pi$ is \textit{pre-cyclic}, if it is of the form $(W_0,W_1,W_1,...)$ for some ``period'' $W_1$ and for some initial finite string of prices $W_0$. 

\begin{prop}\label{prop: Discounted}
For any $\bTau \geq 1$, for any function $f : \Omega^{\bTau} \times \Omega^{\bTau} \to \R^+$, we have the following.
There exists a pre-cyclic $\pi^* \in \cP$ such that,
$$\cR(\pi^*)=\sup_{\pi \in \cP} \cR(\pi).$$
Moreover,  $\pi^*=(W_0,W_1,W_1,\ldots)$, with $W_0=(w^0_1,\ldots,w^0_{P_0})$ and $W_1=(w^1_1,\ldots,w^1_{P_1})$, such that all the coordinates of $(W_0,W_1)$ are distinct, with $P_0+P_1\leq K^\bTau$.
\end{prop}

\begin{proof} 
Decompose any $\pi \in \cP$ as $\pi=(W_0,W_1,\bar{\pi})$ where $W_0=(w^0_1,\ldots,w^0_{P_0})$ and $W_1=(w^1_1=\bar{w},w^1_2,\ldots,w^1_{P_1})$ and $\bar{\pi}=(\bar{w},\ldots) \in \cP$ where $\bar{w} \in  \Omega^{\bTau}$ is the first string that appears twice in the expression for $\pi$. By definition of $\bar{w}$ we have that all the coordinates of $(W_0,W_1)$ are distinct.

Define $\tilde{\pi}=(W_1,\bar{\pi})$. Observe that for some functions $C$ and $C'$ one has
$$\cR(\pi)=C(W_0,\bar w)+e^{-rP_0} \cR(\tilde \pi)$$
as well as
$$\cR(\pi)=C'(W_0,W_1,\bar w)+e^{-r( P_0+P_1)} \cR(\bar \pi).$$

Applying the first equality to an optimal strategy $\pi^*$, we conclude that $\tilde \pi^*$ should maximize $\cR$ among all strategies that start with $\bar{w}^*$. Hence, one can take $\bar{\pi}^*={\tilde \pi}^*$ and still get an optimal strategy $\pi^{**}=(W_0^*,W_1^*,W_1^*,\ldots)$.  $\hfill \square$
\end{proof}
\vspace{0.3cm}


\subsection{Extension to Strategic Customers}\label{sec: Strategic}
We end this section by addressing how the weakly coupled property of the revenue function is also satisfied when customers are strategic in the sense of forward looking.

The decision for a patient customer to purchase the product occurs as soon as the current valuation becomes larger than the current price. In the case of a strategic customer, this decision is more subtle and requires a specific model of strategic behavior. Strategic customers make their purchasing decision not only based on the current valuation and price, but also based on future committed prices as well as  possibly on the distribution of their future valuation (e.g., a customer is expecting a bonus, so the willingness-to-pay will be affected by the bonus amount). Independently of the specifics of the strategic behavior model, we find in the (BP) model that Theorems~\ref{thm: Main} and \ref{thm: Algo} remain valid for strategic customers. Every arriving customer has a bounded patience level $\tau$. So, if we consider a policy of the form, $\pi=\big((w_j:j\geq 1):w_j\in\Omega^{\tau}\big)$,  the customers who arrive between time $j\tau+1$ and $(j+1)\tau$ will generate revenues that are only function of $w_j$ and $w_{j+1}$  regardless of how their strategic behavior is modeled (as long as the stationarity assumption holds). Hence the weakly coupled property holds as in the case of strategic customers and the results of Theorem~\ref{thm: Main} and \ref{thm: Algo} apply.

We end this discussion with two examples of possible strategic models that could govern forward-looking customers.




{\bf Example 3.}  Customers present in the system purchase in a specific period if the current surplus is larger than all possible surplus in the future calculated using the current valuation. Suppose a customer is currently present in the system  with valuation $v_k$ and could remain in the system for at most another $t$ periods. The customer observes the current and future committed prices in his remaining window say, $(k_s, k_{s+1}, ...,k_{s+t})$ and purchase the product at the current price, $k_s$, if $k-k_s\geq 0$ and $k_s\leq \min\{k_{s+1},...,k_{s+t}\}$.  The customer would delay a purchase if the product becomes cheaper in the ``future". In this example, the customer believes that his current valuation is the best prediction of her future valuation.\medskip

{\bf Example 4.} In this example, the customers of every given class have a common model for their stochastic valuation evolution process, call it $\tilde{\mathcal V}_{t}(\cdot)$ (it does not need to be the same as the firm's model).
Customers purchase in a specific period if the current surplus  is positive  and larger than the expected value of the maximum surplus that can be generated in the remaining periods of the patience level, that is if $k-k_s\geq 0$ and 
$$k-k_s\geq \mathbb{E}\max_{s+1\leq j\leq s+t}\{\tilde{\mathcal{V}}_t(j)-k_{j}\}.$$

\section{Numerical Analysis}\label{sec: Numerics}

Our objective in this section is to use the algorithm of Theorem \ref{thm: Algo} to shed light on the behavior of optimal pricing strategies. We confine our analysis to the (M-BP) and (M-UP) models where we have a closed form of the revenue function. We use the algorithm to obtain and analyze an optimal cyclic policy
for many choices of the matrix $Q$, while also varying $\bTau$, the valuation vector ${\bf v}$ and its size $K$. Without much loss of generality we set $\boldsymbol\gamma=[1/K,\ldots,1/K]$.  The parameters we use in drawing different realizations of $Q$ and ${\bf v}$ follow:

\begin{itemize}
\item[i.)]The parameter $\nu>0$ represents the probability in any given period that a customer leaves the system by reaching the absorbing state $v_0$. Without a great loss of generality we often restrict, for simplicity of notations, our numerical analysis to the case where $\nu$ is the same for any given state of the customer. The coefficients of the matrix $Q$ can thus be drawn randomly in $(0,1)$ such that the sum on each row is equal to $1-\nu$.

\item[ii.)]The parameter $e>0$ is introduced to represent the average gap between consecutive valuations. That is, the vector ${\bf v}=[v_1,\ldots,v_K]$ is such that $v_1=1$ and $v_{i+1}=v_i\,(1+e \,\zeta_i)$ where $\zeta_i$ is uniformly drawn in $(0,1)$.
\end{itemize}

When the findings with respect to the (M-BP) and (M-UP) cases are the same, we present these under the (M-BP) setting. We later discuss the findings that are specific to the (M-UP) case. For now, we set $M=1$, or equivalently,  $\sigma=\tau$. For each set of parameters $(\bTau, \nu, e,K)$ we make $n\gg1$ drawings of the different parameters to which we apply the analytic algorithm. We collected and summarized the statistics of the outcomes in Table~1 below.
\begin{itemize}
\item We call $k$-cyclic policy a cyclic and simple $\tau$-policy of size $k\,\tau$ (i.e., with $k$ different prices).

\item We say that a cyclic policy $\pi$ outperforms $\pi'$ by $\delta \%$ if $100(\cR(\pi)-\cR(\pi'))/\cR(\pi')=\delta$.

\item Let $\mf$ be the percentage of times in which the maximal policy performs better than all the fixed-price policies. Let $\md$ denote the average ratio (in percent) of how much the maximal policy outperforms the fixed-price policy (this average is computed conditional on the fact that  the maximal policy is not a fixed-price).

\item Let $\mf'$ be the percentage of times in which the maximal policy performs better than all the fixed-price and two-cyclic policies. Let $\md'$ denote the average ratio (in percent) of how much the maximal policy outperforms the maximum over the  fixed-price and two-cyclic policies (this average is computed conditional to the fact that the maximal policy has period larger or equal to three).
\end{itemize}
\bigskip

\noindent {\bf Finding I. Optimal policies are almost always cyclic policies with a decreasing cycle of length at most three.}

Under (M-BP), each customer experiences no more than two prices; this is true with probability larger than $1-\varepsilon$ in the (M-UP) case. Hence, it is reasonable to expect the optimal cyclic policies given by Theorem~\ref{thm: Main} to have a short cycle size, even if $K$ is large. We know based on Proposition~\ref{prop: Example_basic_intro} that this cannot always be true. However, our numerical analysis actually yielded, in all the cases we considered, optimal policies that are monotone decreasing and $k$-cyclic with $k$ smaller or equal than $3$. Hence, our first finding is that for all values of the parameters and for all the instances we considered (as described above), the optimal policies were decreasing and $k$-cyclic with $k\leq3$.
\bigskip

\noindent {\bf Finding II. Fixed price and two-cyclic policies are practically optimal.}

Our second finding is that the fixed-price and two-cyclic policies were nearly optimal, with very few instances in which a three-cyclic policy outperformed them; and in those cases the performance ratio was less than $1 \%$. This is reflected in the two last columns of Table 1.

\bigskip

\noindent {\bf Finding III. As $\sigma$ increases fixed-price policies perform increasingly better compared to two-cyclic policies.}

In line with our  result of Propositions \ref{prop: Affine_Reduction} we find that, as long as the probability to reach the state $v_0$ or $v_K$ from all the other valuations is supposed to be strictly positive, then as $\sigma$ gets very large the fixed-price policies become the best. This is because customers tend to see just one price as $\sigma$ goes to infinity. As shown by the fourth  column of Table 1, the frequency at which  two-cyclic policies outperform the fixed-price policies is decreasing with $\sigma$. As shown by the fifth column of Table 1, we see the outperforming rate of the two-cyclic policies $\md$ is also a decreasing function of $\sigma$. Finally, the last line of these two columns confirms our result of Proposition \ref{prop: Affine_Reduction}, namely that as $\sigma$ goes to infinity the optimal policy is a fixed-price.

\bigskip

\noindent {\bf Finding IV. As $\nu$ increases fixed-price policies perform increasingly better compared to two-cyclic policies.}

As the probability to reach the states $v_0$ or $v_K$ from all the other valuations increases, the fixed-price policies become the best.  This occurs because as latter probability grows, the customers spend less time in the system. If each customer stays in the system for just one period, then the price that optimizes the revenue of a single period (without consideration of $\Theta$) gives rise to an optimal policy  that is the fixed-price policy with this same price.
 As shown by the fourth  column of Table 1, the frequency at which  two-cyclic policies outperform the fixed-price policies is decreasing with $\nu$. As shown by the fifth column of the Table 1, we find the outperforming rate of the two-cyclic policies $\md$ is also a decreasing function of $\nu$.

\bigskip
\begin{center}
\begin{tabular}{|c|c|c|c|c|c|c|}
\hline
$\tau$ & $\nu$ & $e$ & $\mf$ & $\md$   & $\mf'$ & $\md'$  \\
\hline

1  & 0.1 & 0.1  & 76  &   2.6 & 5  &   0.1  \\
 5 & 0.1 & 0.1  &  43 & 0.6    &  2 &   0  \\
 9 & 0.1  & 0.1  &  21 &  0.3  & 0 &    0\\
 13 & 0.1  & 0.1  &  13 &   0 & 0  &   0  \\

\hline

5 & 0. & 0.1  &   67 &   3      &  5 & 0.3 \\
5 & 0.1 & 0.1  &  54   &  0.6 & 4 & 0.1 \\
5 & 0.2 & 0.1  &   16 &   0.1 & 0 & 0  \\
5 & 0.3 & 0.1  &   1  &    0  & 0 & 0 \\

\hline

5 & 0.1 & 0.05 & 45  & 0.3  & 2 &  0.1 \\
5 & 0.1 & 0.15  &  43  &  0.7 & 3 &  0.1\\
5 & 0.1 & 0.25  &  29  &   1.2  & 2 & 0.2\\
5 & 0.1 & 0.35  &  34  &   1.7  & 2 & 0.2\\
\hline

\end{tabular}

\bigskip

Table 1. \footnotesize{Comparing under an (M-BP) model, with $M=1$ and $K=4$, the optimal policy to optimal fixed-price policies and two-cyclic policies.}

\end{center}

\bigskip

\noindent {\bf Finding V.  Impact of disregarding changing valuations on revenues.}

We illustrate these findings with specific examples. First, consider the following. Set $K=4$ and $\bTau=\sigma=10$ under the (M-BP) case and set the transition matrix $Q$ close to the Identity,

\[
Q= \begin{bmatrix}
  0.9& 0.05& 0.03& 0.02 \\
  0.05& 0.8& 0.1& 0.05\\
  0& 0.05& 0.95& 0\\
  0& 0& 0& 1
 \end{bmatrix}.
\]
Note that $\nu(Q)=1$ as is the case for the Identity matrix. Let  ${\bf v}= [1.0, 1.3, 1.45, 1.6]$.
For this valuation vector and transition matrix $Q$, an optimal policy is the two-cyclic policy $(3,2)$. If however the matrix is the Identity (i.e., the valuations are not changing), an optimal policy is given by $(2,1)$.  Of course, if the seller disregards the impact of the changing valuations, and thus solves for an optimal policy using $Q'={\rm Id}$, she will end up setting the two-cyclic policy $(2,1)$ instead of $(3,2)$. The loss in profit of doing so is found numerically to be $4 \%$. This result is quite robust. Indeed, the policy $(2,1)$ is numerically found to be optimal for all diagonal matrices that are close to the ${\rm Id}$. Thus, the seller would incur the same loss of $4 \%$ if he replaces $Q$ by any diagonal matrix close to the Identity. 

Now, let ${\bf v} = [1.0, 1.1, 1.25, 1.4]$
and
\[
Q= \begin{bmatrix}
  0.7& 0.1& 0.05& 0.07 \\
  0.05& 0.7& 0.07& 0.07\\
  0.07& 0.07& 0.7& 0\\
  0& 0& 0.05& 0.8
 \end{bmatrix}.
\]

For any transition matrix that is diagonal, $Q'=Diag(d_1,d_2,d_3,d_4)$ where $d_i \in [0.6,1]$, we find that the fixed-price policy $(1)$ is optimal. On the other hand, when we consider $Q$ to be the transition matrix, the policy  $(3,2,1)$ is optimal  and outperforms $(1)$ by $5 \%$.
\bigskip

\noindent {\bf Finding VI.  An example of three-cyclic optimal policy. }

In the case of $\bTau=M \sigma$, $M \geq 2$, one expects that optimal policies with longer cycles appear.
However each customer sees $M+1$ prices in this context and it is reasonable to conjecture that the length of the cycles will be small if $M$ is taken to be small. Because of the complexity $K^{4M}$ we are bound to let $M$ be small if we want to apply our algorithm.

In all the examples we tested ($M\leq 4$) the optimal policies were decreasing and had a cyclic length no more than $3$. However, as opposed to the case $M=1$, we often experienced cases where a three-cyclic policy outperforms both two-cyclic and the fixed-price policies by more than $5 \%$. We illustrate this finding with the following example.

Consider the following (M-BP) example with $K=4$, $\bTau=6$ and $\sigma=2$, that is $M=3$. We set ${\bf v}=[1.0, 1.2, 1.4, 2]$ and $\boldsymbol\gamma=[0.25, 0.25, 0.25, 0.25]$. We let
\[
Q= \begin{bmatrix}
  0.8& 0.05& 0.0& 0 \\
  0.2& 0.6& 0.0& 0.2\\
  0.1& 0.2& 0.6& 0\\
  0.1& 0.2& 0.6& 0
 \end{bmatrix}.
\]

\noindent The maximum fixed-price policy  is given by $(1)$ with $\cR((1))=1$. The optimal two-cyclic policy is given by $(3,1)$ with $\cR((3,1))=1.09$. Finally, an optimal policy is a three-cyclic policy and is  given by  $(4,3,1)$ with  $\cR((4,3,1))=1.14$. It is an example where an optimal policy outperforms the fixed-price policies by $14\%$ and the two-cyclic ones by $4 \%$.

\bigskip
\noindent {\bf Finding VII. Tuning $\varepsilon$ in the (M-UP) case. }

Consider the (M-UP) case. We fix $K,Q,{\bf v},$ and $\boldsymbol\gamma$. One could select $\varepsilon$ as a way to obtain the best possible solution of problem (\ref{Eq: General_Main_Form}).
For every $\varepsilon>0$ we define $\tau_\varepsilon=-\ln \varepsilon / \nu$ and we set  $\bTau=\tau_\varepsilon$ and take $\sigma=\bTau$ or equivalently $M=1$ (a similar reasoning below can be made for any $M$).

We denote by  $\pi_\varepsilon$ an optimal policy given by the algorithm that uses the analytic approximation of the revenue by a $\bTau$-affine revenue function. We denote by $\bcR(\pi_\varepsilon)$ the approximate return of an optimal policy as given by the algorithm. Given that the error of the affine approximation is bounded by $\varepsilon$ we may introduce the quantity $\bcR(\pi_\varepsilon)-\varepsilon$ and focus on the choices of $\varepsilon$ that make it positive.

For example, to compare the performance of  $\pi_\varepsilon$ versus the fixed-price policies, we can introduce the notation $\Delta_\varepsilon=\bcR(\pi_\varepsilon)-\varepsilon - \cR^*_f$ where $\cR^*_f$ denotes the revenue of the best fixed-price policy. Typically, as $\varepsilon \to 0$, fixed-price policies are optimal as confirmed by the (M-UP) version of Proposition~\ref{prop: Affine_Reduction} (included in Proposition~\ref{optimal_results}) and the algorithm (see Table 2). As $\varepsilon$ becomes large, we often find two-cyclic optimal policies $\pi_\varepsilon$. However as $\varepsilon$ becomes too large, $\Delta_\varepsilon$ becomes negative which implies that any optimal  policy we find is not guaranteed to perform better than the fixed-price ones, once we account for the error $\varepsilon$ of our approximation. In conclusion, as we vary $\varepsilon>0$, we aim to maximize $\bcR(\pi_\varepsilon)-\varepsilon$ and track $\Delta_\varepsilon$ as a way to compare the approximation with fixed-price policies.

We illustrate this approach using the following example. We let ${\bf v}=[1.0, 1.2, 1.5, 2]$ and
\[
Q= \begin{bmatrix}
  0.65& 0.05& 0& 0.05 \\
  0.05& 0.65& 0& 0.05\\
  0.05& 0.05& 0.6& 0\\
  0& 0& 0& 0.6
 \end{bmatrix}.
\]

Note that $\nu(Q)=0.3$. From Table~2, we observe that there is a broad range of values of $\varepsilon$, where a two-cyclic policy is strictly better than fixed-price policies. In particular, by setting $\varepsilon=0.01$ the algorithm generates a policy that outperforms the best fixed-price policy by more than $3\%$, and represents our best approximation of the unconstrained optimization.

\bigskip
\begin{center}
\begin{tabular}{|c|c|c|c|c|c|c|}
\hline
$\varepsilon$ & $\tau_\varepsilon$ & $\pi_\varepsilon$ & $\bcR(\pi_\varepsilon)$ & $\Delta_\varepsilon $  \\
\hline
$10^{-13}  $& 83  & (4)  & 1.02 &  0    \\
 $10^{-5} $& 32 & (4,1)  &   1.03 &  0.02     \\
 $10^{-3} $ & 19  & (4,1)&  1.04  &   0.04    \\
 $10^{-2} $ & 12  & (4,1)&  1.05  &    0.05   \\
 $10^{-1}$  & 6 & (4,1)&  0.99  &    0   \\
 0.2  & 4  & (4,1)&  0.91  &    -0.1   \\
 0.4  & 2 & (4,1)&  0.73  &    -0.27   \\
\hline
\end{tabular}

\bigskip

Table 2. \footnotesize{Comparing under an (M-UP) model with $M=1$ and fix $(K,Q,{\bf v},\boldsymbol\gamma)$,  for different values of $\varepsilon$, the performance of the near optimal policy, $\pi_\varepsilon$, to the optimal fixed-price policy.}

\end{center}

\section{Conclusion}\label{sec: Conclusion}
 We addressed an intertemporal pricing problem that incorporates what we find to be a natural consumer's behavior; namely, customers with time-varying valuations. We considered a continuous flux of patient customers with random but bounded patience levels during which their valuation processes stochastically change through time. The model depicting the valuation process is extremely general and requires only time stationarity.  We show, with some explicit numerical examples, that disregarding stochastic valuations may lead to  firms missing out on sizable profits. As for the analytical analysis, and despite the generality of the model, the additive revenue function is shown in this context to preserve the weakly coupled property, which guarantees that cyclic policies are optimal. We thus established that the optimality of cyclic policies in the context of intertemporal pricing is quite robust.

The time-varying valuations feature of the model presents a major complexity with respect to the structure and the tractability of these policies. Concerning the latter, one can obtain well-behaved and numerically tractable policies by limiting  the pricing pace, i.e., by forcing the firm to keep each price for a minimum amount of time. As for the structure, previous work (see, \cite{LiuCooper15}) has highlighted the complexity of finding the optimal policy even for constant valuations as soon as one considers multiple patience levels. In the latter setting, \cite{Lobel17} shows that it can remain tractable. However with general time-varying valuations, we find that it is not possible to reduce the exponential complexity of the optimal policy. For that, we constructed an example that shows that even when customers spend no more than two periods in the system ($M=1$), an optimal policy can be cyclic increasing and the size of the cycle can be as large as $K-1$. Interestingly however, we observed numerically that in the same setting of the example ($M=1$), when we randomly select the primitives of the problem, it is rare for an optimal cyclic policy to have a cycle size strictly larger than two (and hence, independent of $K$),  or not monotone \textit{decreasing}. 

In terms of potential next steps, we believe that the last observation is worth investigating carefully by looking for conditions under which the cycle size is independent of the number of prices. This problem is complex, but a better understanding of this possible independence in $K$ would represent a valuable breakthrough.   It will be also valuable to look more carefully into the strategic customer case and select the best way to model such behavior. However, unlike what was observed in \cite{LiuCooper15} and \cite{Lobel17}, we do not see how in the context of changing valuation, the strategic customer case would be easier to handle than the patient case. Finally, most of the closely related literature have been looking at a single product. It could be an interesting challenge to look at substitutable products and see whether one could identify the structure of optimal policies under such a more complex setting.


\theendnotes

\break\section*{Appendix A.}\label{Appen_A}
\setcounter{section}{0}


\section{Markovian Valuations - Formulation of the Revenue Pair}\label{sec: Markov Calc}
This section is devoted to some preliminary analysis and calculations when the valuation process follows a discrete time Markov chain on $\Omega$ with an absorbing state at $v_0=0$; hence its transition matrix can be written as follows
\[ \bar Q=\left( \begin{array}{ccc}
1 & 0 \\
H & Q \\
\end{array} \right),\]
where $H$ is a column vector with dimension $K$, which entries are the probabilities to go from any state $k>0$ to $v_0$. We also assume, without loss of generality, that one customer arrives every period.

Our objective is to characterize how the expected number of customers that are in the system evolves through time and we also obtain a closed form formulation of the revenue function  $(L,\boldsymbol\Theta)$.
\bigskip 

We start by introducing new notations. We set the following
\begin{itemize}
\item[$\bullet$]$P_{i,j^+}=\sum_{l=j}^{K} q_{il}$. The probability to transfer from state $i$ to any possible state $l\geq j$.
\item[$\bullet$]$P_{i^-,j^+}=\big(\sum_{l=j}^{K} q_{1l},\sum_{l=j}^{K} q_{2l}, \ldots, \sum_{l=j}^{K} q_{(i-1)l}\big)^T$. The column vector whose entries are the probabilities to go from a state strictly lower than $i$ to all possible states higher than $j$.
\item[$\bullet$]$U_{i,j^-}=\big(q_{i1},\ldots,q_{i(j-1)}\big)$. The row vector whose entries are the probabilities to go from state $i$ to a specific state strictly lower than $j$
\item[$\bullet$]$Q_{i^-j^-}= \{q_{kl}:1\leq k<i,1\leq l<j\}$ is the matrix of transfer probabilities from states lower than $i$ to states lower than $j$. It is a minor of the matrix $\bar Q$. We denote by $Q_m$ the square matrix $Q_{m^-m^-}$
\end{itemize}
Note first that $``-"$ involves states strictly lower while $``+"$ involves states larger or equal; secondly, all the above quantities involve transition probabilities between states in $\Omega$ and do not involve the absorbing state $v_0$.

\bigskip

\noindent {\bf Expected Number of Customers in the System.}

\noindent We consider $\boldsymbol\theta=(\theta_1,...,\theta_K)\in[0,\tau]^K$, where $\theta_m$ is again the expected number of individuals with valuation $v_m$ who are currently in the system.
Given a phase $(k,\tauu)$, an expected state of the system, $\boldsymbol\theta$, we recall that $\boldsymbol\Theta(k,\tauu|\boldsymbol\,\theta)=\big(\Theta_{1}(k,\tauu|\boldsymbol\theta),...,\Theta_{K}(k,\tauu|\,\boldsymbol\theta)\big)$, is the expected state of the system one phase later. Therefore, if $m\geq k$ then $\Theta_{m}(k,\tau|\,\boldsymbol\theta)\equiv 0,$ otherwise,

\begin{equation}\Theta_{m}(k,\tauu|\,{\boldsymbol\theta})=  \left( \sum_{l=1}^{k-1}\theta_l\, U_{l,k^-} Q_{k}^\tauu+ \sum_{i=0}^{\tauu\wedge\tau-1}\sum_{l=1}^{k-1} \gamma_l \, U_{l,k^-} Q_k^{i}\right)_{m} + \gamma_{m}\label{Eq: Theta_m}
\end{equation}
We recognize that for $\tauu\geq \tau$ and for $m<k$, we have that
\begin{equation}
\Theta_m(k,\tauu)=\bar{\Theta}_{m}(k):=\left( \sum_{i=0}^{\tau-1}\sum_{l=1}^{k-1} \gamma_l \, U_{l,k^-} Q_k^{i}\right)_{m} + \gamma_{m},\label{Eq_Theta_bar}
\end{equation}
and $\Theta_m(k,\tauu)=\bar{\Theta}_{m}(k)\equiv0$, when $m\geq k$.

\bigskip

\noindent {\bf Expected Revenues.}

\noindent We denote by $\cL_{(k,\tauu)}(m)$ the expected revenue generated in the next $\tauu$ periods by one individual currently present in the system with valuation $v_m$, facing in the next $\tauu$ periods the price $v_k$. It is possible that the individual does not purchase during these $\tauu$ periods; any revenues generated after $\tauu$ are not included in $\cL_{(k,\tauu)}(m)$. Therefore,
\begin{equation}
\begin{aligned}
\cL_{(k,\tauu)}(m)
&=\big[P_{m,k^+} + U_{m,k^-} P_{k^-,k^+} +  U_{m,k^-} Q_{k} P_{k^-,k^+} + U_{m,k^-} Q_{k}\, Q_{k}\,P_{k^-,k^+} +\ldots \big]\cdot v_k\\
    &=\big[P_{m,k^+} + \sum_{l=0}^{\tauu\wedge\tau-2}U_{m,k^-}\,Q_{k}^l \,P_{k^-,k^+}\big]\cdot v_k
\end{aligned}\end{equation}
We also denote by $\bar{\cL}_{(k,\tauu)}(m)$ the expected revenue generated from one individual who just arrived into the system with valuation $v_m$, and who will be facing the phase $(k,\tauu)$. Therefore,
\begin{equation*}
\bar{\cL}_{(k,\tauu)}(m)= \delta_{m\geq k}  v_{k} + (1-\delta_{m\geq k}) \cL_{(k,\tauu)}(m).
\end{equation*} The first term depicts the expected revenue generated if at arrival the customer has a valuation larger or equal than the listed price $k$. Otherwise, the customer now is in the system and the expected revenue generated from the remaining period are given by $\mathcal{L}$ defined above.

Putting together some of these formulations we obtain a closed formulation of the expected revenue $L$.
\begin{prop}\label{prop: Closed form Revenues}
Given a transition matrix, $Q$, a vector $\boldsymbol\gamma$ of initial proportions and a set $\Omega$ of valuations, the expected revenues generated from a set of consecutive phases $\big((k_j,\tauu_j):1\leq j\leq n)\big)$, is given by
\begin{equation}
\begin{aligned}
L(k_1,...,k_n:\tauu_1,...,\tauu_n)&=\sum_{j=1}^n L(k_j,\tauu_j|\,\bft^{j-1})\\
&=\sum_{j=1}^n\sum_{m=1}^K \big[\gamma_m\,\sum_{\tauu=1}^{\tauu_j} \bar{\cL}_{(k_j,\tauu)}(m)+\theta^{j-1}_m\,\cL_{(k_{j},\tauu_{j})}(m)\big]
\end{aligned}
\end{equation}
with $\theta_m^j=\Theta_{m}(k_j,\tauu_j|\,{\boldsymbol\theta^{j-1}})$ is given in closed form in (\ref{Eq: Theta_m}) and $\bft^0=\bf{0}.$
\end{prop}

We end this section by a result that is relevant primarily in the unbounded case.

\begin{prop}\label{prop: Bounded_Theta}
Suppose that the system starts empty. Then, for any policy $\pi\in\cP$, and any $1\leq m\leq K$, the expected number of customers in the system with valuation $v_m$ is bounded by $\rho=\min\{\tau,1/\nu\}$.
\end{prop}

The result is obvious in the (M-BP) case. For the (M-UP), it is enough to prove that for any policy $\pi=\big((k_j,\tau_j): j\geq 1\big)\in\cP$, we have that for all $j\geq 1$ and $1\leq m\leq K$, $$\theta_m^j=\boldsymbol \Theta_m(k_j,\tau_j|\,\bft^{j-1}) \leq \rho<\infty.$$ We start with two observations. First, the probability that a customer with valuation $v_m$ jumps next to $v_0$ is upper bounded by $\Lambda:=1-\nu$. Second, the expected number of customers with valuation $v_m$ at time $t$, is always less than the expected number of customers with valuation $v_m$ at time $t$, that would have been in the system if the pricing policy was set at $K$ throughout. Recall, that $\boldsymbol\Theta_K=0$ and $\bft^0=\bf{0}$, and thus from Equation~\ref{Eq: Theta_m} we have that
\begin{align*}
||\bfT(K,t)||&\leq ||\sum_{i=0}^{t-1}\sum_{l=1}^{K-1} \gamma_l \, U_{l,K^-} Q_K^{i}|| + ||(\gamma_{1},\gamma_{2},...,\gamma_{K-1},0)||\\
&\leq \sum_{i=0}^{t-1}\sum_{l=1}^{K-1} \gamma_l\,||U_{l,K^-}||\, ||Q_K||^{i}+1\\
&\leq \Lambda\,\sum_{i=0}^{t-1}\Lambda^{i}+1\leq \frac{1}{1-\Lambda}.
\end{align*} The third inequality is due to the fact that $\sum_{l=1}^{K-1} \gamma_l\,||U_{l,K^-}||\leq \Lambda \sum_l\gamma_l\leq \Lambda.$ The bound above shows that each component of $\bfT$ is bounded by $1/(1-\Lambda)$. Hence, the expected number of customers in the system is bounded by $K/(1-\Lambda)$. 
$\square$\\

\section{Proof of Propositions~\ref{prop:Eps_Affine}}
\subsection{Additional Notations}
We introduce some additional notations. Recall that  $\cL_{(k,\tauu)}(m)$ (resp. $\bar{\cL}_{(k,\tauu)}(m)$) is the expected revenues generated in the next $\tauu$ periods from one individual who is already in the system (resp. just arrived to the system) facing the price $v_k$ for $\tauu$ consecutive periods. Given a certain duration $\tau$, we denote by
\begin{itemize}
\item $C_{k}(m)=\cL_{(k,\tau)}(m)$, is the expected revenues generated during the phase $(k,\tau)$ by an individual initially in state $m$,
\item $B_{k}(m)=\bar{\cL}_{(k,\tau)}(m)$, is the expected revenues generated during the phase $(k,\tau)$ by an individual arriving with valuation $v_m$,
\item $A_k(m)= \sum_{\tauu=1}^{\tau} \bar{\cL}_{(k,\tauu)}(m)$
\item $A_k:= \sum_{m} \gamma_m A_k(m)$, is the aggregate expected revenues generated during the phase, $(k,\tau)$, by all the customers that arrived during this same phase,
\item $B_k:= \sum_{m} \gamma_m B_k(m)$, is the aggregate expected revenues generated during the phase, $(k,\tau)$, by all the customers that newly arrived at the beginning of this phase,
\item $T_{k,k'}:= \sum_{m} \bar\Theta_m(k)C_{k'}(m)$, is the expected revenues generated by ${\bf \bar\Theta}(k)$ during a phase $(k',\tau)$, where, ${\bf \bar\Theta}$ is given by (\ref{Eq_Theta_bar})
\item $ p_{k,\tauu,m}$ is the probability that a customer who faces a phase $(k,\tauu)$ remains in the system at the end of the phase.
\end{itemize}
\subsection{The proof}
Next we detail the proof of Proposition~\ref{prop:Eps_Affine}. Property (A1) holds as a direct result of $\tau<\infty$. As for (A2), we start with the quantity $L(k,\tauu|\,\boldsymbol\theta)$ for $\tauu\geq \tau$ and that can be written as the sum of three terms. First, the expected revenues generated during the phase from new customers in the last $\tau$ periods of the phase. That is exactly $A_{k}$. Secondly, the expected revenues generated from new comers that arrived in the first $\tauu-\tau$ periods of the phase and that is exactly, $\sum_m\sum_{s=\tau+1}^{\tauu}\gamma_m\,\bar{\mathcal{L}}_{(k,s)}(m)$. Finally, the revenues generated from customers that were in the system at the beginning of the phase, $\boldsymbol\theta$, and that is $\sum_m\theta_m\,\mathcal{L}_{k,\tauu}(m).$ We write that
\begin{align*}
L(k,\tauu|\boldsymbol\theta)&=A_{k}+\sum_m\sum_{s=\tau+1}^{\tauu}\gamma_m\,\bar{\mathcal{L}}_{(k,s)}(m)+\sum_m\theta_m\,\mathcal{L}_{k,\tauu}(m)\\
&= A_{k}+\sum_m\sum_{s=\tau+1}^{\tauu}\gamma_m(B_{k}(m)+p_{k,\tau,m}\,v_0)+\sum_m\theta_{m}
\,C_{k}(m)+\sum_m\theta_m\,p_{k,\tau,m}\,v_0\\
&= A_{k}+B_{k}\,(\tauu-\tau)+\sum_m\theta_{m}
\,C_{k}(m).
\end{align*}
By setting $\bold C_k$ to be the vector with entry $C_{k}(m)$, we complete our proof.

\break\section*{Appendix~B}\label{Appen_C}
\section*{A Simpler Model. The case of $K=2$}
 In this section, we consider an (M-UP) model with $K=2$. Let  $\Omega=\{v_1,v_2\}$, with $v_0< v_1\leq v_2$. A seller facing a patient customer would only consider a pricing policy ${\bold p}$ where at any time $t$, $p_t\in \{v_1,v_2\}$. We start by observing that the price $v_1$ is a reset price and that is, once set the system empties. Having this in mind, it is easy to see that the optimal policy is either a  fixed-price policy where for all $t$, $p_t=v_i$ $i=1,2$, or is cyclical of the form $\pi_c=\big((v_2,\tauu),(v_1,1)\big)$; the cycle starts with $v_2$ that is set for $0<\tauu<\infty$ consecutive periods and then $v_1$ is set once. Note that the extreme cases of $\tauu\in\{0,\infty\}$ cover the cases where the prices are constant.

The arriving customers has valuation $v_2$ with probability~$\gamma$, and have valuation $v_1$ with probability $\bar{\gamma}=1-\gamma$. The average revenue per period collected by the seller is $R(\tauu)$, with
\[
\cR(\tauu)={1\over \tauu+1} \left(\gamma v_2 \tauu + v_1 +v_1 {\bar{\gamma} q_{12}\over 1-q_{11}} \left(\tauu-{1-q_{11}^\tauu\over 1-q_{11}}\right)+v_1  \bar{\gamma} {1-q_{11}^\tauu\over 1-q_{11}} (q_{11}+q_{12}) \right)
\]
The first two terms in parenthesis represent the revenues from customers that arrive and buy right away: those with valuation $v_2$ at price $v_2$ during $\tauu$ periods, and everyone at price $v_1$ in the last period. The last two terms in parenthesis account for the customers that accumulate over time and end up buying at either $v_1$ or $v_2$. The expected number of customers buying in a particular period $l$, $1\leq l \leq \tauu$, at price $v_2$ is given by the expression
\[
\bar{\gamma} q_{12} \sum_{n=0}^{l-2} q_{11}^n = \bar{\gamma} q_{12} {1-q_{11}^{l-1}\over 1-q_{11}},
\]
i.e., these are all customers that arrive earlier with low valuation $v_2$, and for whom the valuation remains at $v_2$ and jumps to $v_1$ from at period~$l-1$. Summing over all~$l$, $1\leq l \leq \tauu$, we get:
\[
\sum_{l=1}^\tauu \left( \bar{\gamma} q_{12} {1-q_{11}^{l-1}\over 1-q_{11}}\right)=  {\bar{\gamma} q_{12} \over 1-q_{11}} \left(\tauu - \sum_{l=1}^\tauu q_{11}^{l-1} \right)= {\bar{\gamma} q_{12} \over 1-q_{11}} \left(\tauu - \sum_{l=0}^{\tauu-1} q_{11}^{l} \right) = {\bar{\gamma} q_{12} \over 1-q_{11}}\left(\tauu - {1-q_{11}^\tauu\over 1-q_{11}}\right)
\]
The expected number of customers buying in period $\tauu+1$ at price $v_1$ is given by those who arrived earlier with valuation $v_1$ and remained at $v_1$ until period $\tauu$, and then either stayed at $v_1$ or jumped to $v_2$, i.e.,
\[
\bar{\gamma}  \sum_{l=0}^{\tauu-1} q_{11}^l (q_{11}+q_{12}) = \bar{\gamma} (q_{11}+q_{12})  {1-q_{11}^\tauu\over 1-q_{11}}.
\]
We next state our main result here where we show that under some conditions involving the primitives of the problem, the seller is better off implementing a strictly cyclic policy of the form $\pi_c=\big((v_2,\tauu),(v_1,1)\big)$ with $\tauu$ finite.
\begin{prop}
We denote by $\tauu^*=\arg\max_{\tauu\in\mathbb{N}}\{\cR(\tauu)\}$. We have the following
$$0<\tauu^*<\infty~~~~~~~~~~~~\text{iff}~~~~~~~~~~C\ln\,q_{11}<B-v_1<C,$$ where $$B=\gamma\,v_2+v_2\bar{\gamma} { q_{12}\over 1-q_{11}}=v_2\,{\gamma q_{20}+q_{12}\over 1-q_{11}}>0,$$
and $$C=-v_2\bar{\gamma}{q_{12}\over(1-q_{11})^2}+v_2\bar{\gamma}{q_{11}+q_{12}\over 1-q_{11}}.$$If these conditions do not hold, then $\tauu^*=0$ iff $B-v_1<C\ln\,q_{11}$, and $\tauu^*=\infty$ otherwise.
\end{prop}
\begin{proof}Given the expressions of $B$ and $C$, we rewrite the revenue function as follows
$$(\tauu+1)\,\mathcal{R}(\tauu)=v_1 +B\,\tauu+C\,(1-q_{11}^\tauu).$$
Note that $C>0$ iff $v_2<v_1(1+{q_{11}\over q_{12}})\,(1-q_{11})$ where the term $(1+{q_{11}\over q_{12}})\,(1-q_{11})\geq 1$.
In order to prove that it is optimal for the seller to set $p_t$ at $v_1$ and reset the system, it is enough to prove that $\tauu^*$ the maximizer of $\cR$ is finite i.e. $0\leq \tauu^*<\infty$.  
Simple calculations show that $$\mathcal{R}'(\tauu)={B-v_1-C+C\,q_{11}^\tauu\,(1-(\tauu+1)\,\ln(q_{11}))\over(\tauu+1)^2}.$$
We denote by $\cal N(\tauu)$ the numerator of $\cal{R}'(\tauu)$. By taking the derivative with respect to $\tauu$ of $g(\tauu)=q_{11}^\tauu\,(1-(\tauu+1)\,\ln(q_{11}))$, we observe that this quantity is non-increasing in $\tauu$, from which we conclude that $\cal N(\tauu)$ is itself monotone in $\tauu$. Therefore, the equation $\mathcal{R}'(\tauu)=0$ admits at most one solution. In light of that, we denote by $\hat{\tauu}^*$ the supremum of $\cR$ on the positive real line and conclude that a necessary and sufficient condition for  $0<\hat{\tauu}^*<\infty$ is that $$i.)~ \mathcal{N}(0)>0 ~~~\text{and}~~~~ ii.)~ \mathcal{N}(\infty)<0.$$  Hence, by putting these two inequalities together we get that  $-C(1-\ln q_{11})<B-v_1-C$ and $B-v_1-C<0$. We finally get that $$0<\hat{\tauu}^*<\infty~~~~~~~~~~~~\text{iff}~~~~~~~~~~C\ln\,q_{11}<B-v_1<C.$$ Note that under these condition $\cR(\tauu)> \cR(\infty)\equiv B$ for any $\tauu\geq \hat{\tauu}^*$ and $\cR(\tauu)> \cR(0)$ for any $\tauu\leq \hat{\tauu}^*$ and hence $\tauu^*:=\arg \max_n\cR(n)\in(0,\infty)$, which completes our proof.
\end{proof}

$$ $$

\noindent {\sc Acknowledgments.} The authors are grateful to Thierry Bousch for pointing out the reference Karp (1978), to Jiawei Zhang for a discussion of one of the main results of the paper, and the associate editor for his/her comments that helped improve the initial version of this paper. The first author was visiting Insead and LPSM Jussieu during the first writing of the paper and he thanks these institutions for their very warm hospitality. The second author was supported by the ANR-15-CE40-0001.

$$ $$ 

\bibliographystyle{ormsv080}
\bibliography{BiblioPricing}

\begin{thebibliography}{24}
\expandafter\ifx\csname natexlab\endcsname\relax\def\natexlab#1{#1}\fi
\expandafter\ifx\csname url\endcsname\relax
  \def\url#1{{\tt #1}}\fi
\expandafter\ifx\csname urlprefix\endcsname\relax\def\urlprefix{URL }\fi
\expandafter\ifx\csname urlstyle\endcsname\relax
  \expandafter\ifx\csname doi\endcsname\relax
  \def\doi#1{doi:\discretionary{}{}{}#1}\fi \else
  \expandafter\ifx\csname doi\endcsname\relax
  \def\doi{doi:\discretionary{}{}{}\begingroup \urlstyle{rm}\Url}\fi \fi

\bibitem[{Ahn et~al.(2007)Ahn, G\"{u}m\"{u}\c{s}, and Kaminsky}]{Ahnetal07}
Ahn, H., M.~G\"{u}m\"{u}\c{s}, P.~Kaminsky. 2007.
\newblock Pricing and manufacturing decisions when demand is a function of
  prices in multiple periods.
\newblock {\it Oper. Res.\/} {\bf 55}(6) 1169--1188.

\bibitem[{Aviv and Pazgal(2008)}]{Aviv08}
Aviv, Y., A.~Pazgal. 2008.
\newblock Optimal pricing of seasonal products in the presence of
  forwardlooking consumers.
\newblock {\it Manufacturing Service Oper. Management\/} {\bf 10} 339--359.

\bibitem[{Aviv and Vulcano(2012)}]{AvivVulc12}
Aviv, Y., G.~Vulcano. 2012.
\newblock A partially observed markov decision process for dynamic pricing.
\newblock {\it Chapter 23, in Handbook of Pricing Management, (Eds. O. Ozer and
  R. Phillips). Oxford University Press\/}  522--584.

\bibitem[{Besanko and Winston(1990)}]{Besanko90}
Besanko, D., W.~Winston. 1990.
\newblock ptimal pricing skimming by a monopolist facing rational consumers.
\newblock {\it Management Sci.\/} {\bf 36}(5) 555--567.

\bibitem[{Besbes and Lobel(2015)}]{Besbes15}
Besbes, O., I.~Lobel. 2015.
\newblock Intertemporal price discrimination: Structure and computation of
  optimal policies.
\newblock {\it Management Sci.\/} {\bf 61}(1) 92--110.

\bibitem[{Bitran and Caldentey(2003)}]{Bitran03}
Bitran, G., R.~Caldentey. 2003.
\newblock An overview of pricing models for revenue management.
\newblock {\it MSOM\/} {\bf 5} 203--229.

\bibitem[{Board(2008)}]{Board2008}
Board, S. 2008.
\newblock Periodic pricing of seasonal products in retailing.
\newblock {\it Review of Economic Studies\/} {\bf 75}(2) 391--413.

\bibitem[{Caldentey et~al.(2016)Caldentey, Liu, and Lobel}]{Caldentey16}
Caldentey, R., Y.~Liu, I.~Lobel. 2016.
\newblock Intertemporal pricing under minimax regret.
\newblock {\it forthcoming, Oper. Res.\/} .

\bibitem[{Coase(1972)}]{Coase72}
Coase, R. 1972.
\newblock Durability and monopoly.
\newblock {\it Journal of Law \& Economics\/} {\bf 15}(1) 43--49.

\bibitem[{Conlisk et~al.(1984)Conlisk, Gerstner, and Sobel}]{Conlisk84}
Conlisk, J., E.~Gerstner, J.~Sobel. 1984.
\newblock Cyclic pricing by a durable goods monopolist.
\newblock {\it Quarterly Journal of Economics\/} {\bf 99} 489--505.

\bibitem[{Deb(2014)}]{Deb14}
Deb, R. 2014.
\newblock Intertemporal price discrimination with stochastic values.
\newblock {\it Working Paper, University of Toronto\/} .

\bibitem[{Gallego and \c{S}ahin(2010)}]{GallegoSahin10}
Gallego, G., \"{O}. \c{S}ahin. 2010.
\newblock Revenue management with partially refundable fares.
\newblock {\it Oper. Res.\/} {\bf 58}(4) 817--833.

\bibitem[{Garrett(2016)}]{Garrett16}
Garrett, D. 2016.
\newblock Intertemporal price discrimination: dynamic arrivals and changing
  values.
\newblock {\it forthcoming American Economic Review\/} .

\bibitem[{Hu et~al.(2016)Hu, Chen, and Hu}]{Huetal16}
Hu, Z., X.~Chen, P.~Hu. 2016.
\newblock Technical note—dynamic pricing with gain-seeking reference price
  effects.
\newblock {\it Oper. Res.\/} {\bf 64}(1) 150--157.

\bibitem[{Levin et~al.(2009)Levin, McGill, and Nediak}]{Levin09}
Levin, Y., J.~McGill, M.~Nediak. 2009.
\newblock Dynamic pricing in the presence of strategic consumers and
  oligopolistic competition.
\newblock {\it Management Sci.\/} {\bf 55}(1) 32--46.

\bibitem[{Levin et~al.(2010)Levin, McGill, and Nediak}]{Levin10}
Levin, Y., J.~McGill, M.~Nediak. 2010.
\newblock Optimal dynamic pricing of perishable items by a monopolist facing
  strategic consumers.
\newblock {\it Prod. Oper. Manage.\/} {\bf 19}(1) 40--60.

\bibitem[{Liu and Cooper(2015)}]{LiuCooper15}
Liu, Y, W.L. Cooper. 2015.
\newblock Optimal dynamic pricing with patient customers.
\newblock {\it Oper. Res.\/} {\bf 63}(6) 1307--1319.

\bibitem[{Lobel(2017)}]{Lobel17}
Lobel, I. 2017.
\newblock Dynamic pricing with heterogeneous patience levels.
\newblock Working Paper, Stern School of Business, New York University, NY.

\bibitem[{Shu and Su(2007)}]{Shen07}
Shu, Z.~M., X.~Su. 2007.
\newblock Customer behavior modeling in revenue management and auctions: A
  review and new research opportunities.
\newblock {\it Production and Operations Management\/} {\bf 16} 713--728.

\bibitem[{Sobel(1981)}]{Sobel91}
Sobel, J. 1981.
\newblock Durable goods monopoly with entry of new consumers.
\newblock {\it Econometrica\/} {\bf 59}(5) 11455--1485.

\bibitem[{Stockey(1979)}]{Stockey79}
Stockey, N. 1979.
\newblock Intertemporal price discrimination.
\newblock {\it Quarterly Journal of Economics\/} {\bf 93}(3) 355.

\bibitem[{Stockey(1981)}]{Stockey81}
Stockey, N. 1981.
\newblock Rational expectations and durable goods pricing.
\newblock {\it The Bell Journal of Economics\/} {\bf 1}(12) 112--128.

\bibitem[{Su(2007)}]{Su07}
Su, X. 2007.
\newblock Intertemporal pricing with strategic customer behavior.
\newblock {\it Management Sci.\/} {\bf 53}(5) 726.

\bibitem[{Wang(2016)}]{Wang16}
Wang, Z. 2016.
\newblock Technical note—intertemporal price discrimination via reference
  price effects.
\newblock {\it Oper. Res.\/} {\bf 64}(2) 290--296.

\end{thebibliography}

\end{document}